\documentclass[reqno]{amsart}
\usepackage[utf8x]{inputenc}
\usepackage{amsmath, amsthm, amssymb}
\usepackage[usenames,dvipsnames,svgnames,table]{xcolor}
\usepackage{enumerate}
\usepackage{dsfont}
\usepackage{cleveref}
\usepackage{cite}
\usepackage{comment}
\usepackage{bbm}
\usepackage{comment}

\newtheorem{theorem}{Theorem}[section]

\newtheorem{lemma}[theorem]{Lemma}
\newtheorem{definition}[theorem]{Definition}

\newtheorem{assumption}[theorem]{Assumption}

\newcommand{\PP}{\mathbb P}

\newcommand{\divv}{\text{div } }

\def\comma{ {\rm ,\qquad{}} }

\newcommand{\be}{\begin{equation}}
\newcommand{\ee}{\end{equation}}

\numberwithin{equation}{section}

\title[Stochastic Boussinesq equations]{Global existence for the stochastic Boussinesq equations with transport noise and small rough data}
\author{Quyuan Lin}
\address{Department of Mathematics, University of California, Santa Barbara, CA 93106}
\email{quyuan\_lin@ucsb.edu}
\author{Rongchang Liu}
\address{Department of Mathematics, University of Arizona, Tucson, AZ 85721}
\email{lrc666@math.arizona.edu}
\author{Weinan Wang}
\address{Department of Mathematics, University of Arizona, Tucson, AZ 85721}
\email{weinanwang@math.arizona.edu}
\date{\today}

\begin{document}
\begin{abstract}
    In this paper, we consider the stochastic Boussinesq equations on $\mathbb T^3$ with transport noise and rough initial data. We first prove the existence and uniqueness of the local pathwise solution with initial data in $L^p(\Omega;L^p)$ for $p>5$. By assuming additional smallness on the initial data and the noise, we establish the global existence of the pathwise solution. 
\end{abstract}

\maketitle



\section{Introduction}
We are interested in the stochastic Boussinesq equations on a 3D torus $\mathbb T^3$ 
\begin{align}\label{e.w09241}
    \begin{split}
    du&=(\Delta u-\mathcal P(u\cdot \nabla u -\rho e_{3}))dt+\mathcal P(b\cdot\nabla u + \sigma^{(1)}(u,\rho)) d\mathbb{W}_t, \quad \nabla\cdot u =0, \\
    d\rho&=(\Delta \rho -u\cdot \nabla\rho )dt + (b\cdot\nabla \rho + \sigma^{(2)}(u,\rho))d\mathbb{W}_t.
    \end{split}
\end{align}
Here $u=u(x, t)$ is the velocity vector field, $\rho = \rho(x, t)$ is the scalar temperature or density of fluid, $\mathcal P$ is the Leray projection (see \eqref{leray}), and $\mathbb{W}_t$ is a cylindrical Wiener process valued in some separable Hilbert space and the corresponding stochastic integral is in the Itô sense. The term $(\mathcal P(b\cdot\nabla u), b\cdot\nabla\rho)d\mathbb{W}_t$ represents the transport noise and $(\mathcal P(\sigma^{(1)}(u,\rho)),\sigma^{(2)}(u,\rho))d\mathbb{W}_t$ is the multiplicative noise. The main result of the paper is the global well-posedness of the system \eqref{e.w09241} in $L^p(\Omega;L^p(\mathbb T^3))$ with $p>5$ and small initial data.

The deterministic Boussinesq system describes the evolution of the velocity field $u$ of an incompressible fluid under the buoyancy $\rho e_3$. Recently, there has been a lot of progress made on the existence, persistence of regularity, and long time behavior of solutions, mostly in the case of positive viscosity. Mathematically, the 2D Boussinesq system
is closely related to the 3D incompressible Euler and Navier-Stokes equations since they share a similar vortex stretching effect. In the case of zero viscosity and diffusivity, the 2D Boussinesq equations can be used as a proxy for the 3D axis-symmetric Euler
equation with swirl away from the symmetric axis. For more recent well-posedness and long time behavior results for the deterministic Boussinesq equations with fractional or full dissipation, cf. \cite{ BH,CN, HL, KW1,LLT}. See \cite{KW2, W1} when boundary conditions are imposed.

Over the past few decades, there has been a growing interest in investigating the impact of stochastic effects on fluid models. By introducing white noise terms into the system and assuming random initial data, these models can account for both numerical and empirical uncertainties. This approach can yield predictions that not only reflect a realistic trajectory but also provide insight into associated uncertainties. Along this line of research, Bensoussan and Temam started the study on the stochastic Navier-Stokes equations \cite{BT}. Mikulevicius and Rozovskii \cite{MR} addressed global $L^2$ well-posedness. See also \cite{ZZ} for local well-posednes for the 3D stochastic Navier-Stokes equations. More recently, Kukavica, Xu, and Ziane considered the stochastic Navier-Stokes system with multiplicative noise and rough initial data \cite{KXZ}. For the stochastic Boussinesq equations, Du addressed local well-posedness for the 3D Boussinesq system with Sobolev initial data \cite{D}. Duan and Millet \cite{DM} studied large deviation principle of the stochastic Boussinesq system. Pu and Guo \cite{PG} proved global well-posedness for the 2D Boussinesq equations with additive white noise while F\"{o}ldes et al \cite{FGRT} considered ergodic and mixing properties of the {B}oussinesq equations. In \cite{WY}, together with Yue, the third author of this paper proved almost-sure global existence of weak solutions to the Boussinesq equations using random data approach. Luo \cite{L} recently considered 2D Boussinesq equations with transport noise and Alonso-Or\'{a}n and Bethencourt de Le\'{o}n proved global well-posedness with transport noise and Sobolev initial data \cite{AB}. See also \cite{HZSG} for the stochastic fractional Boussinesq equations. For the recent progress for other stochastic fluid euations, see, for example, \cite{brzezniak2001stochastic,capinski1999stochastic,glatt2014local} for the stochastic Euler equations and \cite{brzezniak2021well,debussche2011local,debussche2012global,hu2022local,hu2023pathwise} for the stochastic primitive equations.

To the best of our knowledge, this work seems to be the first one to consider the global well-posedness for the 3D Boussinesq equations with transport noise and rough data. The transport noise represent a wider and more general class of noise than the multiplicative noise.
In the study of turbulent flows, the transport noise has been introduced in \cite{kraichnan1968small,kraichnan1994anomalous} and widely studied in many different SPDEs, see for example \cite{agresti2022stochastic-1,agresti2022stochastic-2,mikulevicius2004stochastic,MR,flandoli2008introduction,agresti2021stochastic-2,brzezniak1992stochastic} and refereces therein. It is worthwhile to mention that with transport noise we are able to prove the existence and uniqueness of solutions when $(u_0,\rho_0)\in L^p(\Omega;L^p(\mathbb T^3))$ with $p>5$. The extension to $p>3$ (see \cite{kukavica2023local} for Navier-Stokes equations with multiplicative noise) remains open when the transport noise is present. Our result is not covered by \cite{agresti2021stochastic-2}, which investigates the Navier-Stokes equations in Besov spaces, as $L^p$ spaces and Besov spaces are not equivalent.

Compared to previous results where people consider either multiplicative noise or smoother initial data, one of the main difficulties in this work is on the rough $L^p$ analysis of transport noise under Leray projection on the velocity field. Indeed, in the $L^p$ scenario, the regularity of the dissipation is at a scale weaker than the $W^{1, p}$. Therefore, one cannot control the energy of the transport noise through a $W^{1, p}$ estimate. In the case when there is no Leray projection, by assuming smallness on the noise intensity, one can obtain an energy estimate of the transport noise that is comparable with (and hence compensated by) the dissipation. However, this approach is not directly applicable in the presence of Leray projection due to the non-local feature of the projection. Our key observation is that the energy of the orthogonal part of the Leray projection on the transport noise $\nabla\Delta^{-1}\divv(b\cdot \nabla u)d\mathbb{W}_t$,
when combined with the divergence free property of the velocity $u$, can be controlled by $L^p$ energy of the velocity through standard elliptic regularity estimates. This enables us to obtain a desired energy estimate for the Leray projection of the transport noise at a scale comparable with dissipation and $L^p$ energy. Another difficulty is to control the convection term during the fixed point iteration process. This is overcome by using a double cut-off trick introduced in \cite{KXZ}. The local pathwise solution is then obtained as the limit of solutions to the truncated systems by sending the truncating scale to infinity. Fixing the cut-off at a small scale, global existence is proved by analyzing the truncated system, whose solution agrees with the solution to \eqref{e.w09241} over all time horizon with high probability as long as the initial data is small. This involves an analysis of the transport noise mentioned earlier that requires its intensity to be sufficiently small.


 The rest of this paper is organized as follows. In Section \ref{NP}, we introduce basic notations and preliminaries. In Section \ref{s.L022301} we establish the $L^p$ estimates for the linear parabolic equation with transport noise under Leray projection. Such estimate is the key to show the local and global well-posedness for the stochastic Boussinesq equations. Section \ref{s.w02271} is devoted to the study of a truncated system from the Boussinesq equations, and after that by applying a family of suitable stopping times we are able to obtain the maximal pathwise solution to the original Boussinesq equations. Finally, we prove the global existence of pathwise solutions in Section \ref{sec:global}.

\section{Notations and preliminaries}\label{NP}

\subsection{Functional settings}
The universal constant $C$ appears in the paper may change from line to line. We shall use subscripts to indicate the dependence of $C$ on other parameters when necessary, {\it e.g.}, $C_r$ means that the constant $C$ depends only on $r$. 

Let $L^p$ and $W^{s,p}$ with $s\in \mathbb N$ and $p\geq 1$ be the usual Sobolev spaces (see \cite{adams2003sobolev}). Let $B_{p, q}^{s}$ with $q\geq1$ be  the usual Besov space (see \cite{bahouri2011fourier}). In the periodic setting, for $s\in \mathbb N$ the $n-$th Fourier coefficient of an $L^1$ function $f$ on $\mathbb T^3$ is defined as
\begin{equation}
    \mathcal Ff(n)=\hat{f}(n)=\int_{\mathbb T^3}f(x)e^{-2\pi in\cdot x}\,dx\comma n\in\mathbb Z^3.
\end{equation}
Denote the operator $J^s$ by
\begin{equation}
    J^{s}f(x)=\sum_{n \in \mathbb Z^3}(1+4\pi^2 |n|^2)^{s/2}
    \hat{f}(n)e^{2\pi i n\cdot x}
    \comma x\in \mathbb T^3\comma s\in \mathbb R.
\end{equation}
In the periodic setting, we have the following equivalent $W^{s,p}$ norm: 
\begin{align*}
    \frac1C \|J^s f\|_{p} \leq  \|f\|_{W^{s,p}} \leq C\|J^s f\|_{p} \comma s\geq 0\comma 1<p<\infty,
\end{align*}
where the usual $L^p$ norms are denoted by $\|\cdot \|_p$.

Next, we define the Leray projector
\begin{equation}\label{leray}
    \mathcal P u = u - \nabla \Delta^{-1} \nabla \cdot u.
\end{equation}
Here $\Delta^{-1}$ is defined subject to periodic boundary condition with zero mean. Denote by $\mathbf{P} = (\mathcal{P}, \mathrm{id})$.
We can rewrite system \eqref{e.w09241} in the following abstract form:
\begin{equation}\label{e.L10261}
    dU = (AU + B(U) + G(U))dt + \left(\mathbf{P}(b\cdot\nabla U) + \sigma(U)\right)d\mathbb{W},
\end{equation}
where $U = (u, \rho):=(U_1,U_2,U_3,U_4)$, $AU = (\Delta u, \Delta\rho)$, $B(U):=-\mathbf{P}(u\cdot \nabla U)$, $G(U) = \mathbf{P}(\rho e_3, 0) = \mathbf{P}(U_4 e_3, 0)$, $\sigma(U) = (\mathcal P \sigma^{(1)}(U), \sigma^{(2)}(U))$. Notice that when $\int U_0\,dx =0$, $\nabla\cdot b=0$, and $\sigma$ satisfies condition \eqref{e.Q031101} below, one can integrate system \eqref{e.L10261} in $\mathbb T^3$ to obtain that
$\int U\,dx =0$ for all $t\geq 0.$ For convenience, we denote by
\begin{align*}
    W_{sol}^{s,p} = \left\{ \mathbf{P}f : f\in W^{s,p}, \int_{\mathbb T^3} f\,dx = 0 \right\}.
\end{align*}

\subsection{Stochastic preliminaries}
We denote by $\mathcal{H}$ a real separable Hilbert space with a complete orthonormal basis $\{\mathbf{e}_k\}_{k\geq 1}$. $(\Omega, \mathcal{F},(\mathcal{F}_t)_{t\geq 0},\mathbb{P})$ represents a complete probability space with an augmented filtration $(\mathcal{F}_t)_{t\geq 0}$.  With $\{W_k: k\in\mathbb N\}$  a family of independent $\mathcal{F}_t$-adapted Brownian motions, $\mathbb W( t,\omega):=\sum_{k\geq 1} W_k( t,\omega) \mathbf{e}_k$ is an $\mathcal{F}_t$-adapted and $\mathcal{H}$-valued cylindrical Wiener process.

For a real separable Hilbert space $\mathcal{Y}$, we define $l^2( \mathcal{H},\mathcal{Y})$ to be the set of Hilbert-Schmidt operators from $\mathcal{H}$ to $\mathcal{Y}$ with the norm  defined by   
\begin{equation}    
\Vert G\Vert_{l^2( \mathcal{H},\mathcal{Y})}^2:= \sum_{k=1}^{\dim \mathcal{H}} | G \mathbf{e}_k|_{\mathcal{Y}}^2<\infty \comma G\in l^2( \mathcal{H},\mathcal{Y}).      
\end{equation} 
In this paper, we either regard \eqref{e.L10261} as a vector-valued equation or consider it component-wise. Correspondingly, $\mathcal{Y}={\mathbb R}$ or ${\mathbb R}^{d}$. Let $G=(G_1, \cdots, G_d)$ and $G\mathbf{e}_k:=(G_1\mathbf{e}_k, \cdots, G_d\mathbf{e}_k)$. Then $G\in l^2( \mathcal{H},{\mathbb R}^{d})$ if and only if $G_i\in l^2( \mathcal{H},{\mathbb R})$ for all $i\in\{1,\cdots, d \}$.

Next, we denote by   
\begin{equation}     
\mathbb{W}^{s,p}:=
\left\{f:\mathbb T^3\to l^2( \mathcal{H},\mathcal{Y}):              f\mathbf{e}_k\in W^{s,p}(\mathbb T^3) \text{~for each~}k, \text{~and~} \int_{\mathbb T^3} \Vert J^s f\Vert_{l^2( \mathcal{H},\mathcal{Y})}^p \,dx<\infty                       \right\},
\end{equation} 
 with respect to the norm    
 \begin{equation}    
 \Vert f\Vert_{\mathbb{W}^{s,p}}:=\left( \int_{\mathbb T^3} \Vert J^s f\Vert_{l^2( \mathcal{H},\mathcal{Y})}^p \,dx\right)^{1/p}.       \end{equation} 
 Furthermore, we denote $(J^s f) \mathbf{e}_k=J^s (f \mathbf{e}_k)$. In particular, $\mathbb{W}^{0,p}$ is abbreviated as $\mathbb{L}^{p}$. Letting $(\mathbf{P} f) \mathbf{e}_k=\mathbf{P}\left(f \mathbf{e}_k\right)$, we have $\mathbf{P} f \in \mathbb{W}^{s, p}$ if $f \in \mathbb{W}^{s, p}$. Write
$$
\mathbb{W}_{\text {sol }}^{s, p}=\left\{\mathbf{P} f: f \in \mathbb{W}^{s, p},  \int_{\mathbb T^3} f\, dx = 0\right\}.
$$
For each sufficiently regular $f: (t, x)\in\mathbb{R}_+\times\mathbb{T}^3\to\mathcal{H}\otimes\mathbb{R}^3$, denote
\[N_{f,k} = \|f\|_{L^{\infty}(\mathbb{R}_+;W^{k,\infty}(\mathbb{T}^3,l^2))}^2 = \sum_{n=1}^{\dim\mathcal{H}}\|f_n\|_{L^{\infty}_tW^{k,\infty}_x}^2.\]
Finally, the Burkholder-Davis-Gundy (BDG) inequality
\begin{equation}\label{e.L10263}    
\mathbb E \left[ \sup_{s\in[0,t]}\left| \int_0^s G \,d\mathbb W_r \right|_{\mathcal{Y}}^p\right]      \leq         C_{BDG} \mathbb E\left[ \left(\int_0^t \Vert G\Vert^2_{ l^2( \mathcal{H},\mathcal{Y})}\, dr \right)^{p/2}\right]      
\end{equation} 
holds for $p\in [1,\infty)$ and all $G\in l^2( \mathcal{H},\mathcal{Y})$ such that the right hand side is finite. 

\subsection{Assumptions on the noise}
\begin{assumption}\label{a.L022201}
We impose the following conditions to the noise coefficients $\sigma, b$:
\begin{enumerate}
\item General conditions on $\sigma, b$: 
\begin{enumerate}
\item \be\label{e.L10264}
\sum_{j=1}^4\left\|\sigma_j(U)\right\|_{\mathbb{L}^p} \leq C\left(\|U\|_{p}+1\right), 
\ee
\be\label{e.L121807}
\sum_{j=1}^4\left\|\sigma_j(U)-\sigma_j(V)\right\|_{\mathbb{L}^p} \leq C\|U-V\|_{p},
\ee
\be\label{e.Q031101}
\sigma\left(W_{\text {sol }}^{s, p}\right) \subset \mathbb{W}_{\text {sol }}^{s, p}.
\ee
\item For each $n$, $b_n(t,x):=b\mathbf{e}_n(t, x):\mathbb{R}_+\times\mathbb{T}^3\to\mathbb{R}^3$ is measurable and $\nabla\cdot b_n(t, \cdot)=0$ for every $t$. In addition $N_{b, 2}<\infty$ and the following super-parabolic condition is satisfied: 
there exists $\nu>0$ such that 
\begin{align*}
\sum_{i, j = 1}^3\left(\varepsilon_{ij} - \frac12\sum_{n=1}^{\dim\mathcal{H}}b_n^i(t, x)b_n^j(t, x)\right)\xi_i\xi_j\geq \nu\|\xi\|^2, \quad (t, x,\xi)\in \mathbb{R}_+\times \mathbb{T}^3\times \mathbb{R}^3, 
\end{align*}
where $\varepsilon_{ij}=1$ if $i=j$ and $0$ otherwise.  
\end{enumerate}
\item For local existence, we assume in addition: 
\begin{align}\label{e.L021403}
N_{b, 0} <\frac{p-1}{2(p-1)+p C_{B D G}^2}.
\end{align}
\item For global existence we assume additionally:
\begin{enumerate}
\item Superlinearity condition: 
\be\label{e.L022205}
\sum_{i=1}^4\left\|\sigma_i(U)\right\|_{\mathbb{L}^p} \leq \epsilon_0\|U\|_{p}, \text{ where } \epsilon_0 \text{ is small.}
\ee
\item Smallness of the transport noise: $N_{b, 2}$ is sufficiently small. 
\end{enumerate}
\end{enumerate}
\end{assumption}

\subsection{Definitions of solutions and main results}
The followings are the definitions for the local pathwise solution and maximal pathwise solution for the Boussinesq equations.

\begin{definition}[Local pathwise solution for the Boussinesq equations]\label{d.w03092}
Fix $T>0$. A pair $(U,\tau)$ is called a local pathwise solution to system \eqref{e.L10261} on $(\Omega, \mathcal{F},(\mathcal{F}_t)_{t\geq 0},\mathbb{P})$ if $\tau$ is a stopping time with $\PP(\tau>0)=1$ and if $U\in L^p(\Omega; C([0,\tau\wedge T], L^p)$ is progressively measurable and it satisfies   
\begin{align}    
\begin{split}      
(U(t),\phi)      &= (U_0,\phi)           +\int_0^t \left[ (U, A\phi)+  (G(U),\phi) + (\mathbf{P}(u\otimes U), \nabla \phi) \,\right]dr         
\\
&\hspace{1cm}+\int_0^t \left[\bigl(\sigma(U)),\phi\bigr) - (\mathbf{P}(b\otimes U), \nabla \phi) \right]dW(r)   ,    
\end{split}  
\end{align} 
$\PP\mbox{-a.s.} $ for all $\phi\in C^{\infty}(\mathbb T^3, \mathbb R^4)$. 
\end{definition}

\begin{definition}[Maximal pathwise solution for the Boussinesq equations]
     Fix $T>0$. A pair $(U, \tau)$ is a maximal pathwise $L^p$ solution to system \eqref{e.L10261} on $(\Omega, \mathcal{F},(\mathcal{F}_t)_{t\geq 0},\mathbb{P})$ if there exists an increasing sequence of stopping times $\tau_n$ with $\tau_n \uparrow \tau$ a.s. such that each pair $\left(U, \tau_n\right)$ is a local pathwise solution,
$$
\sup _{0 \leq t \leq \tau_n}\|U(t)\|_p^p+\sum\limits_{j=1}^4 \int_0^{\tau_n} \int_{\mathbb{T}^3}\left|\nabla\left(|U_j(t)|^{p / 2}\right)\right|^2 d x d t<\infty,
$$
and
$$
\sup _{0 \leq t \leq \tau}\|U(t)\|_p^p+\sum\limits_{j=1}^4 \int_0^{\tau_n} \int_{\mathbb{T}^3}\left|\nabla\left(|U_j(t)|^{p / 2}\right)\right|^2 d x d t=\infty,
$$
on the set $\{\tau \leq T\}$.
\end{definition}

The following is the main theorem concerning the local and global existence of pathwise solutions to the system \eqref{e.L10261}.
\begin{theorem}\label{t.w10101}
Assume that (1) and (2) in the Assumption \eqref{a.L022201} hold. Let $p>5$ and $U_0\in L^{p}(\Omega; L^p (\mathbb T^3))$ satisfying $\int_{\mathbb T^3} U \, dx=0$. Then there exists a unique maximal pathwise solution $(U, \tau)$ to \eqref{e.L10261} such that
\begin{equation}
\begin{split}
    &\mathbb E \left[\sup_{0\leq s\leq \tau}\|U(s)\|_{p}^{p}
    +
    \int_{0}^{\tau} \sum_{j=1}^{4}\int_{\mathbb T^3}|\nabla(|U_{j}(s,x)|^{p/2})|^2 \,dxds
    \right]
    \leq
    C_p\mathbb E[\|U_{0}\|_{p}^{p}+1],
    \end{split}
\end{equation}

If additionally (3) in the Assumption \eqref{a.L022201} holds, then for every $\epsilon\in(0,1]$ there is $\delta>0$ such that if $\mathbb{E}\|U_0\|_p^p\leq\delta$, then 
\begin{align}\label{L.022501}
\mathbb{P}(\tau=\infty)\geq 1-\epsilon.
\end{align}
\end{theorem}

\section{$L^p$ estimates for the linear equation}\label{s.L022301}
Consider the stochastic linear parabolic equation with transport noise
\begin{align}\label{e.L021401}
\begin{split}
&dU = (\Delta U + GU+f)dt + (\mathbf{P}\left(b\cdot\nabla U \right) + g)d\mathbb{W}_t, \\
&U(0) = U_0, 
\end{split}
\end{align}
where $U = (u, \rho)$, $GU = (\mathcal{P}(\rho e_3), 0)$ and $\mathbf{P}(b\cdot \nabla U) = (\mathcal{P}(b\cdot\nabla u), b\cdot\nabla \rho)$. We denote by $f:=(\overline f,f_4)$ and $g:=(\overline g,g_4)$ where $\overline f=(f_1,f_2,f_3)$ and $\overline g = (g_1, g_2, g_3)$ representing the first three components of $f$ and $g$, respectively. We also assume that $\nabla\cdot \overline f = \nabla\cdot \overline g = 0$. Here $f$ and $g$ are independent of $u$ and $\rho$. We first give a result concerning smooth data.
\begin{lemma}\label{l.L01221}
Let $p>2$. Assume the conditions on $b$ as in Assumption \ref{a.L022201} and 
\[U_0\in L^p(\Omega; W^{3,p}), f\in L^p(\Omega\times(0, T); W^{1,p}), g\in L^p(\Omega\times(0, T); \mathbb{W}^{2,p}).\]
Then there is a unique pathwise solution $U$ to \eqref{e.L021401} on $[0, T]$ such that 
\begin{align}\label{e.L020401}
U\in L^p(\Omega; C([0, T]; W^{2, p})).
\end{align}
\end{lemma}
\begin{proof}
The condition on $U_0$ implies that $U_0\in L^p(\Omega;B_{p,p}^{3-2/p})$ due to the embedding $W^{3,p}\hookrightarrow B_{p,p}^{3-2/p}$. As $f$ and $g$ are independent of $u$ and $\rho$, the equation for $U_4$ (which is $\rho$) is independent of $u = (U_1, U_2, U_3)$: 
\begin{align}\label{e.L021202}
dU_4 = (\Delta U_4 + f_4)dt +(b\cdot\nabla U_4 + g_4)d\mathbb{W}_t, \quad U_{4}(0) = U_{0,4}.
\end{align}
Under the conditions of Lemma \ref{l.L01221}, we can apply Theorem 5.2 in \cite{AV21c} (with $\kappa=0, p=q$) to deduce  the existence of a unique pathwise solution $U_4$ on $[0, T]$ to equation \eqref{e.L021202} such that $U_4\in L^p(\Omega; L^p((0, T]; W^{3, p}))\cap L^p(\Omega; C([0, T]; B_{p, p}^{3-2/p}))$. 
Let $ \widetilde{f} = \overline f + \mathcal{P}(U_4 e_3)$. The equation for the velocity $u$ then becomes 
\begin{align}\label{e.L021203}
du = (\Delta u + \widetilde{f})dt +(\mathcal{P}(b\cdot\nabla u) + \overline{g})d\mathbb{W}_t, \quad u(0) = (U_{0,1},U_{0,2},U_{0,3}).
\end{align}
In view of the conditions of Lemma \ref{l.L01221}, for equation \eqref{e.L021203} we can apply Theorem 3.2 in \cite{agresti2021stochastic-2} (with $\kappa=0, p=q,h=0$) to obtain a unique pathwise solution $u$ on $[0, T]$ to equation \eqref{e.L021203} such that 
$u\in L^p(\Omega; L^p((0, T]; W^{3, p}))\cap L^p(\Omega; C([0, T]; B_{p, p}^{3-2/p}))$.
The conclusion of the lemma then follows from  the embedding $B_{p,p}^{3-2/p}\hookrightarrow W^{2, p}$.
\end{proof}

Next, we consider rough condition for system \eqref{e.L021401}.

\begin{theorem}\label{t.L021401}
Assume the conditions (1b) and (2) on $b$ as in Assumption \ref{a.L022201}. Let $p>2$ and $0<T<\infty$. Suppose $U_0\in L^p(\Omega; L^p(\mathbb{T}^3))$, $f\in L^p(\Omega\times[0, T]; W^{-1, q}(\mathbb{T}^3)), g\in L^p(\Omega\times[0, T]; \mathbb{L}^p(\mathbb{T}^3))$, where 
\begin{align}\label{e.L030401}
\frac{3p}{p+1}<q\leq p.
\end{align}
Then there is a unique maximal solution $U\in L^p(\Omega; C([0, T], L^p))$ to equation \eqref{e.L021401} such that 
\begin{align}\label{e.L021402}
\begin{split}
    &\mathbb{E}\left[\sup _{0 \leq t \leq T}\|U(t, \cdot)\|_{p}^{p}+\sum_{j=1}^{4} \int_{0}^{T} \int_{\mathbb{T}^{3}}\left|\nabla\left(\left|U_{j}(t, x)\right|^{p / 2}\right)\right|^{2} d x d t\right] \\
    &\quad \leq  C\mathbb{E}\left[\left\|U_{0}\right\|_{p}^{p}+\int_{0}^{T}\|f(t, \cdot)\|_{-1, q}^{p} d t+\int_{0}^{T} \left\|g(t, \cdot)\right\|_{\mathbb{L}^{p}}^pd t\right],
\end{split}
\end{align}
where $C>0$ depends on $T, p, b$.
\end{theorem}

\begin{proof}
The proof is divided into the following three steps. 
\vskip0.1in
{\it Step 1}. Let 
\[\rho_{\epsilon}(\cdot) = \frac{1}{\epsilon^3}\rho\left(\frac{\cdot}{\epsilon}\right)\]
be a standard mollifier and denote 
\[f^{\epsilon} = f * \rho_\epsilon, \quad g^{\epsilon} = g* \rho_\epsilon, \quad U_0^{\epsilon} = U_0 * \rho_\epsilon,\]
where $*$ represents the convolution.  The smoothness of the mollified objects and Lemma \ref{l.L01221} imply that the following system 
\begin{align}\label{e.L021207}
\begin{split}
&dU^{\epsilon} = (\Delta U^{\epsilon} + GU^{{\epsilon}}+f^{\epsilon})dt + (\mathbf{P}\left(b\cdot\nabla U^{\epsilon} \right)+ g^{\epsilon})d\mathbb{W}_t, \\
&U(0) = U_0^{\epsilon}, 
\end{split}
\end{align}
has a unique pathwise solution $U^{\epsilon}\in L^p(\Omega; C([0, T]; W^{2, p}))$. For each $j\in\{1, 2, 3, 4\}$, it follows from Ito's formula that  
\begin{align}\label{e.L021208}
\begin{split}
    \|U^{\epsilon}_j(t)\|_p^p &+ p(p-1)\int_0^t\int_{\mathbb{T}^3}|U^{\epsilon}_j(r)|^{p-2}|\nabla U^{\epsilon}_j(r)|^2dxdr\\ &= \|U^{\epsilon}_{0, j}\|_p^p 
    + p\int_0^t\int_{\mathbb{T}^3}|U^{\epsilon}_j(r)|^{p-2}U^{\epsilon}_j(r)\left((GU^{\epsilon})_j(r)+ f^{\epsilon}_j(r)\right)dxdr\\
    &+ p\int_0^t\int_{\mathbb{T}^3}|U^{\epsilon}_j(r)|^{p-2}U^{\epsilon}_j(r)\Big(\big(\mathbf{P}\left(b\cdot\nabla U^{\epsilon} \right)\big)_j+ g^{\varepsilon}_j(r)\Big)dxd\mathbb{W}_r\\
    &+ \frac{p(p-1)}{2}\int_0^t\int_{\mathbb{T}^3}|U^{\epsilon}_j(r)|^{p-2}\|\big(\mathbf{P}\left(b\cdot\nabla U^{\epsilon} \right)\big)_j+ g^{\varepsilon}_j(r)\|_{l^2}^2dxdr.
\end{split}
\end{align}

\vskip0.1in
{\it Step 2}. We now deduce estimate \eqref{e.L021402}  for $U^{\varepsilon}$. We first recall the following Poincar\'e type inequalities (see \cite{KZ06,KXZ})
\begin{align}\label{e.L031101}
\begin{split}
    &\left\||v|^{p-1}\right\|_{q} \leq C\left\|\nabla\left(|v|^{p-1}\right)\right\|_{q}, \\
    &\left\|\left|v\right|^{p-2} v\right\|_{q} \leq C\left\|\nabla\left(\left|v\right|^{p-2}v\right)\right\|_{q},\quad p,q>1,
\end{split}
\end{align}
which is valid for zero mean functions $v$ with sufficient regularity so that the expressions make sense. It then follows from \eqref{e.L031101} that 
\begin{align}\label{e.L030402}
\begin{split}
    p\int_0^t\int_{\mathbb{T}^3}|U^{\epsilon}_j(r)|^{p-2}U^{\epsilon}_j(r)&\left((GU^{\epsilon})_j(r)+ f^{\epsilon}_j(r)\right)dxdr\\
    &\leq C\int_0^t\|\nabla(|U^{\epsilon}_j(r)|^{p-2}U^{\epsilon}_j(r))\|_{q'}\|(GU^{\epsilon})_j(r)+ f^{\epsilon}_j(r)\|_{-1, q}dr,
\end{split}
\end{align}
where $1/q+1/q'=1$. Let $\bar{r}$ be such that $1/\bar{r}+1/2=1/q'$, then by H\"older inequality one has
\begin{align}\label{e.L030403}
\begin{split}
    \|\nabla(|U^{\epsilon}_j|^{p-2}U^{\epsilon}_j)\|_{q'}\leq C\left\|\left|U^{\epsilon}_j\right|^{p / 2-1} \nabla\left(\left|U^{\epsilon}_j\right|^{p / 2}\right)\right\|_{q^{\prime}} 
    &\leq C\left\|\left|U^{\epsilon}_j\right|^{p / 2-1}\right\|_{\bar{r}}\left\|\nabla\left(\left|U^{\epsilon}_j\right|^{p / 2}\right)\right\|_2\\
    & = C\left\|\left|U^{\epsilon}_j\right|^{p/ 2}\right\|_{\bar{r}(p-2)/p}^{(p-2)/p}\left\|\nabla\left(\left|U^{\epsilon}_j\right|^{p / 2}\right)\right\|_2.
\end{split}
\end{align}
In view of the condition \eqref{e.L030401}, we have 
\[2 \leq \frac{\bar{r}(p-2)}{p}<6.\]
By Poincar\'e type inequalities \eqref{e.L031101} and Gagliardo-Nirenberg inequality, we deduce 
\begin{align}\label{e.L030404}
    \left\|\left|U^{\epsilon}_j\right|^{p/ 2}\right\|_{\bar{r}(p-2)/p}\leq C\left\|\left|U^{\epsilon}_j\right|^{p/ 2}\right\|_{2}^{1-\alpha}\left\|\nabla\left(\left|U^{\epsilon}_j\right|^{p / 2}\right)\right\|_2^{\alpha},
\end{align}
where $\alpha = \frac32-\frac{3p}{\bar{r}(p-2)}$. Combining inequalities \eqref{e.L030402}-\eqref{e.L030404}, and Young's inequality with $\delta>0$, we obtain 
\begin{align}\label{e.L021404}
\begin{split}
&p\int_0^t\int_{\mathbb{T}^3}|U^{\epsilon}_j(r)|^{p-2}U^{\epsilon}_j(r)\left((GU^{\epsilon})_j(r)+ f^{\epsilon}_j(r)\right)dxdr\\
&\leq C\int_0^t\left\|\left|U^{\epsilon}_j\right|^{p/ 2}\right\|_{2}^{(1-\alpha)(p-2)/p}\left\|\nabla\left(\left|U^{\epsilon}_j\right|^{p / 2}\right)\right\|_2^{1+\alpha(p-2)/p}\|(GU^{\epsilon})_j(r)+ f^{\epsilon}_j(r)\|_{-1, q}dr\\
&\leq \delta \int_0^t\int_{\mathbb{T}^3}|U^{\epsilon}_j(r)|^{p-2}|\nabla U^{\epsilon}_j(r)|^2dxdr +\delta \int_0^t\|U^{\epsilon}_j(r)\|_p^pdr\\
&\hspace{2cm}+C_\delta\left (\int_0^t\left\|f^{\epsilon}_j(r)\right\|_{-1, q}^p d r + \int_0^t\left\|(GU^{\epsilon})_j(r)\right\|_{-1, p}^p d r \right),
\end{split}
\end{align}
where we used the fact 
\[\int_{\mathbb{T}^3}\left|\nabla\left|U^{\epsilon}_j\right|^{p / 2}\right|^2dx = \frac{p^2}{4}\int_{\mathbb{T}^3}|U^{\epsilon}_j(r)|^{p-2}|\nabla U^{\epsilon}_j(r)|^2dx.\]
To deal with the terms involving transport noise, we utilize the following property of Leray projection
\begin{align*}
   \mathcal{P}(b\cdot \nabla u^{\epsilon}) = b\cdot \nabla u^{\epsilon} - \nabla\Delta^{-1}\divv(b\cdot \nabla u^{\epsilon}),\quad  u^{\epsilon} = (U_1^{\epsilon},U_2^{\epsilon},U_3^{\epsilon}).
\end{align*}
As a result, we have 
\begin{align}\label{e.L022101}
    \mathbf{P}\left(b\cdot\nabla U^{\epsilon} \right) = b\cdot\nabla U^{\epsilon} - \mathbf{Q}\left(b\cdot\nabla U^{\epsilon} \right), \quad \mathbf{Q}\left(b\cdot\nabla U^{\epsilon}\right) := \big(\nabla\Delta^{-1}\divv(b\cdot \nabla u^{\epsilon}), 0\big).
\end{align}
Since $u^{\epsilon}$ is divergence free, one has 
\begin{align}\label{e.L022103}
    (\mathbf{Q}\left(b\cdot\nabla U^{\epsilon}\right))_j = \partial_j\Delta^{-1}\sum_{k, \ell=1}^3\partial_kb^{\ell}\partial_{\ell} u^{\epsilon}_{k}, \quad j=1, 2, 3.
\end{align}
To estimate the last term in \eqref{e.L021208}, we first note 
\begin{align}\label{e.L022102}
\begin{split}
     &\frac{p(p-1)}{2}\int_0^t\int_{\mathbb{T}^3}|U^{\epsilon}_j(r)|^{p-2}\|\big(\mathbf{P}\left(b\cdot\nabla U^{\epsilon} \right)\big)_j+ g^{\varepsilon}_j(r)\|_{l^2}^2dxdr\\
     &\leq p(p-1)\int_0^t\int_{\mathbb{T}^3}|U^{\epsilon}_j(r)|^{p-2}\|b\cdot\nabla U^{\epsilon}_j+ g^{\varepsilon}_j(r)\|_{l^2}^2dxdr
     \\
   &\hspace{1cm}+p(p-1)\int_0^t\int_{\mathbb{T}^3}|U^{\epsilon}_j(r)|^{p-2}\|(\mathbf{Q}\left(b\cdot\nabla U^{\epsilon}\right))_j\|_{l^2}^2dxdr
     : = I_1 + I_2.
\end{split}
\end{align}
For $I_1$, we have
\begin{align}\label{e.L121803}
\begin{split}
    I_1&\leq 2p(p-1)N_{b, 0}\int_0^t\int_{\mathbb{T}^3}|U^{\epsilon}_j(r)|^{p-2}|\nabla U^{\epsilon}_j(r)|^2dxdr + 2p(p-1)\int_0^t\int_{\mathbb{T}^3}|U^{\epsilon}_j(r)|^{p-2}\|g^{\epsilon}_j(r)\|_{l^2}^2dxdr\\
    &\leq 2p(p-1)N_{b, 0}\int_0^t\int_{\mathbb{T}^3}|U^{\epsilon}_j(r)|^{p-2}|\nabla U^{\epsilon}_j(r)|^2dxdr + \int_0^t\left\|U^{\epsilon}_j(r)\right\|_p^p d r+C \int_0^t\left\|g^{\epsilon}_j(r)\right\|_{\mathbb{L}^p}^p d r
\end{split}
\end{align}
where we applied Young’s inequality in the last step. For $I_2$, it follows from \eqref{e.L022103} that 
\begin{align}\label{e.L0221013}
\begin{split}
    I_2& = p(p-1)\int_0^t\int_{\mathbb{T}^3}|U^{\epsilon}_j(r)|^{p-2}\|(\mathbf{Q}\left(b\cdot\nabla U^{\epsilon}\right))_j\|_{l^2}^2dxdr\\
    &  = p(p-1)\int_0^t\int_{\mathbb{T}^3}|U^{\epsilon}_j(r)|^{p-2}\sum_{n=1}^{\dim\mathcal{H}}\left|\partial_j\Delta^{-1}\sum_{k, \ell=1}^3\partial_kb_{n}^{\ell}\partial_{\ell} u^{\epsilon}_{k}\right|^2dxdr\\
    &\leq p(p-1)\sum_{n=1}^{\dim\mathcal{H}}\int_0^t\left(\|U^{\epsilon}_j(r)\|_{p}^{p-2}\left\|\partial_j\Delta^{-1}\sum_{k, \ell=1}^3\partial_kb_{n}^{\ell}\partial_{\ell} u^{\epsilon}_{k}\right\|_{p}^2\right)dr
\end{split}
\end{align}
By standard elliptic estimates, we have 
\begin{align*}
       \left\|\partial_j\Delta^{-1}\left(\partial_kb_{n}^{\ell}\partial_{\ell} u^{\epsilon}_{k}\right)\right\|_{p}^2&\leq C\left\|\partial_kb_{n}^{\ell}\partial_{\ell} u^{\epsilon}_{k}\right\|_{W^{-1, p}}^2 
       \\
       &= C \left|\sup\limits_{\|\phi\|_{W^{1,p'}}=1} (\partial_kb_{n}^{\ell}\partial_{\ell} u^{\epsilon}_{k}, \phi) \right|^2= C \left|\sup\limits_{\|\phi\|_{W^{1,p'}}=1} (u^{\epsilon}_{k}, \partial_{\ell} ( \partial_kb_{n}^{\ell} \phi ) )\right|^2 \\
       &\leq  C\|\partial_kb_{n}^{\ell}\|_{L_{t, x}^{\infty}}^2\left\|u^{\epsilon}_{k}\right\|_{p}^2 + C_{E}\|\partial_{\ell}\partial_kb_{n}^{\ell}\|_{L_{t, x}^{\infty}}^2\left\| u^{\epsilon}_{k}\right\|_{p}^2
       \\
       &\leq C\max_{1\leq k, \ell\leq3}\|\partial_kb_{n}^{\ell}\|_{L_{t, x}^{\infty}}^2\left\|U^{\epsilon}\right\|_{p}^2 + C\max_{1\leq k, \ell\leq3}\|\partial_{\ell}\partial_kb_{n}^{\ell}\|_{L_{t, x}^{\infty}}^2\left\|U^{\epsilon}\right\|_{p}^2. 
\end{align*}
Therefore,
\begin{align}\label{e.L022304}
       I_2\leq C(N_{b,1}+N_{b,2})\int_0^t\left\|U^{\epsilon}(r)\right\|_p^p d r.
\end{align}
Combining this with \eqref{e.L022102} and \eqref{e.L121803}, we obtain 
\begin{align}\label{e.L022104}
  \begin{split}
     &\frac{p(p-1)}{2}\int_0^t\int_{\mathbb{T}^3}|U^{\epsilon}_j(r)|^{p-2}\|\big(\mathbf{P}\left(b\cdot\nabla U^{\epsilon} \right)\big)_j+ g^{\varepsilon}_j(r)\|_{l^2}^2dxdr\\
     \leq &2p(p-1)N_{b, 0}\int_0^t\int_{\mathbb{T}^3}|U^{\epsilon}_j(r)|^{p-2}|\nabla U^{\epsilon}_j(r)|^2dxdr 
     \\
     &\hspace{2cm}+ C \int_0^t\left\|U^{\epsilon}(r)\right\|_p^p d r+C \int_0^t\left\|g^{\epsilon}_j(r)\right\|_{\mathbb{L}^p}^p d r.
\end{split}
\end{align}

Next, the decomposition \eqref{e.L022101}, BDG inequality and Minkowski's inequality yield
\begin{align}\label{e.L022105}
\begin{split}
   &\mathbb{E}\sup_{t\in[0, T]}\left| \int_0^t\int_{\mathbb{T}^3}|U^{\epsilon}_j(r)|^{p-2}U^{\epsilon}_j(r)(b(r)\cdot\nabla U^{\epsilon}_j(r) + g^{\epsilon}_j(r))dxd\mathbb{W}_r\right|\\
   &\leq C_{BDG}\mathbb{E}\left(\int_0^T\left\|\int_{\mathbb{T}^3}|U^{\epsilon}_j(r)|^{p-2}U^{\epsilon}_j(r)b(r)\cdot\nabla U^{\epsilon}_j(r)dx\right\|_{l^2}^2dr\right)^{1/2}
\\&\quad+
C_{BDG}\mathbb{E}\left(\int_0^T\left\|\int_{\mathbb{T}^3}|U^{\epsilon}_j(r)|^{p-2}U^{\epsilon}_j(r)g^{\epsilon}_j(r)dx\right\|_{l^2}^2dr\right)^{1/2}\\
   &\quad + C_{BDG}\mathbb{E}\left(\int_0^T\left\|\int_{\mathbb{T}^3}|U^{\epsilon}_j(r)|^{p-2}U^{\epsilon}_j(r)\left(\mathbf{Q}\left(b\cdot\nabla U^{\epsilon}\right)\right)_j dx\right\|_{l^2}^2dr\right)^{1/2}: = J_1 + J_2 + J_3.
\end{split}
\end{align}
 By Minkowski's inequality and Hölder's inequality, we obtain
\begin{align}\label{e.L022107}
\begin{split}
   &J_1 = C_{BDG}\mathbb{E}\left(\int_0^T\left\|\int_{\mathbb{T}^3}|U^{\epsilon}_j(r)|^{p-2}U^{\epsilon}_j(r)b(r)\cdot\nabla U^{\epsilon}_j(r)dx\right\|_{l^2}^2dr\right)^{1/2}\\
   &\leq C_{BDG}\mathbb{E}\sup_{t\in[0,T]}\|U^{\epsilon}_j(t)\|_{p}^{p/2}\left(\int_0^T\int_{\mathbb{T}^3}|U^{\epsilon}_j(r)|^{p-2}\left\|b\cdot \nabla U^{\epsilon}_j(r)\right\|_{l^2}^2dxdr\right)^{1/2}\\
   &\leq \frac{1}{4p}\mathbb{E}\sup_{t\in[0, T]}\|U^{\epsilon}_j(t)\|_{p}^p + pC_{BDG}^2 N_{b,0}\mathbb{E}\int_0^T\int_{\mathbb{T}^3}|U^{\epsilon}_j(r)|^{p-2}|\nabla U^{\epsilon}_j(r)|^2dxdr.
 \end{split}
\end{align}
A similar estimate yields 
\begin{align}\label{e.L022106}
J_2\leq \frac{1}{4p}\mathbb{E}\sup_{t\in[0, T]}\|U^{\epsilon}_j(t)\|_{p}^p+C_T\mathbb{E}\int_0^T\|g^{\epsilon}_j(r)\|_{\mathbb{L}^p}^pdr.
\end{align}
To estimate $J_3$, we denote $h = |U^{\epsilon}_j(r)|^{p-2}U^{\epsilon}_j(r)$ and note that by integration by parts and \eqref{e.L022103}, since $\Delta^{-1}$ is a self-adjoint operator,
\begin{align}\label{e.L022108}
\begin{split}
 \int_{\mathbb{T}^3}h\left(\mathbf{Q}\left(b\cdot\nabla U^{\epsilon}\right)\right)_j dx &= \sum_{k, \ell=1}^3\int_{\mathbb{T}^3}h\partial_j\Delta^{-1}\big(\partial_kb^{\ell}\partial_{\ell} u^{\epsilon}_{k}\big) dx\\
 & = -\sum_{k, \ell=1}^3\int_{\mathbb{T}^3}\Delta^{-1}(\partial_jh )\partial_kb^{\ell}\partial_{\ell} u^{\epsilon}_{k}dx\\
 & = \sum_{k, \ell=1}^3\left(\int_{\mathbb{T}^3}\partial_{\ell} \Delta^{-1}(\partial_jh )\partial_kb^{\ell}u^{\epsilon}_{k}dx + \int_{\mathbb{T}^3} \Delta^{-1}(\partial_jh) \partial_{\ell}\partial_kb^{\ell}u^{\epsilon}_{k}dx\right).
 \end{split}
\end{align}
As a consequence, we have 
\begin{align}\label{e.L022109}
\begin{split}
&\left\|\int_{\mathbb{T}^3}h\left(\mathbf{Q}\left(b\cdot\nabla U^{\epsilon}\right)\right)_j dx\right\|_{l^2}\\
& \leq\sum_{k, \ell=1}^3\left(\int_{\mathbb{T}^3}|\partial_{\ell} \Delta^{-1}\partial_jh |\left(\sum_{n=1}^{\dim\mathcal{H}}|\partial_kb_n^{\ell}|^2\right)^{1/2}|u^{\epsilon}_{k}|dx + \int_{\mathbb{T}^3}| \Delta^{-1}\partial_jh| \left(\sum_{n=1}^{\dim\mathcal{H}}|\partial_{\ell}\partial_kb_n^{\ell}|^2\right)^{1/2}|u^{\epsilon}_{k}|dx\right)\\
&\leq N_{b, 1}^{1/2}\sum_{k, \ell=1}^3\int_{\mathbb{T}^3}|\partial_{\ell}\Delta^{-1}\partial_jh||u^{\epsilon}_{k}|dx + 3N_{b, 2}^{1/2}\sum_{k}^3\int_{\mathbb{T}^3}|\Delta^{-1}\partial_jh||u^{\epsilon}_{k}|dx.
 \end{split}
\end{align}
By Hölder's inequality and standard elliptic regularity estimates, we have for $1/p' +1/p=1$
\begin{align}\label{e.L0221010}
\begin{split}
\int_{\mathbb{T}^3}|\partial_{\ell}\Delta^{-1}\partial_jh||u^{\epsilon}_{k}|dx&\leq \|\partial_{\ell}\Delta^{-1}\partial_jh\|_{L^{p'}}\|u^{\epsilon}_{k}\|_{p}\\
&\leq C\|h\|_{p'}\|U^{\epsilon}\|_{L^{p}} \leq C\|U^{\epsilon}\|_{p}^{p-1}\|U^{\epsilon}\|_{p} =C\|U^{\epsilon}\|_{p}^{p}.
 \end{split}
\end{align}
Similarly we deduce 
\begin{align*}
\int_{\mathbb{T}^3}|\Delta^{-1}\partial_jh||u^{\epsilon}_{k}|dx\leq C\|U^{\epsilon}\|_{p}^{p}. 
\end{align*}
Combining this with \eqref{e.L022109} and \eqref{e.L0221010} one has 
\begin{align}\label{e.L0221011}
\begin{split}
J_3 & = C_{BDG}\mathbb{E}\left(\int_0^T\left\|\int_{\mathbb{T}^3}|h\left(\mathbf{Q}\left(b\cdot\nabla U^{\epsilon}\right)\right)_j dx\right\|_{l^2}^2dr\right)^{1/2}\\
&\leq C( N_{b, 1}^{1/2} + N_{b, 2}^{1/2})\mathbb{E}\left(\int_0^T\|U^{\epsilon}\|_{p}^{2p}dr\right)^{1/2}\\
&\leq \frac{1}{4 p} \mathbb{E} \sup _{t \in[0, T]}\left\|U_j^\epsilon(t)\right\|_p^p + C\mathbb{E}\int_0^T\left\|U^\epsilon\right\|_{p}^{p} d r.
\end{split}
\end{align}
Therefore by combining \eqref{e.L022105}, \eqref{e.L022106}, \eqref{e.L022107} and \eqref{e.L0221011}, we have 
\begin{align}\label{e.L121804}
\begin{split}
   &p\mathbb{E}\sup_{t\in[0, T]}\left| \int_0^t\int_{\mathbb{T}^3}|U^{\epsilon}_j(r)|^{p-2}U^{\epsilon}_j(r)(b\cdot\nabla U^{\epsilon}_j(r) + g^{\epsilon}_j(r))dxd\mathbb{W}_r\right|\\
   &\leq\frac34\mathbb{E}\sup_{t\in[0, T]}\|U^{\epsilon}_j(t)\|_{p}^p + p^2C_{BDG}^2N_{b,0}\mathbb{E}\int_0^T\int_{\mathbb{T}^3}|U^{\epsilon}_j(r)|^{p-2}|\nabla U^{\epsilon}_j(r)|^2dxdr\\
  &\quad+C\mathbb{E}\int_0^T\left\|U^\epsilon\right\|_{p}^{p} d r+C_T\mathbb{E}\int_0^T\|g^{\epsilon}_j(r)\|_{\mathbb{L}^p}^pdr.
\end{split}
\end{align}
Now taking supremum over $[0, T]$ on both sides of \eqref{e.L021208} and integrating over $\Omega$, using inequalities \eqref{e.L021404}, \eqref{e.L022104} and \eqref{e.L121804}, we derive for $j=1,2,3$, 
\begin{align}\label{e.L0221012}
\begin{split}
    &\mathbb{E}\left[\frac14\sup_{t\in[0, T]}\|U^{\epsilon}_j(t)\|_{p}^p +p(p-1)\int_0^T\int_{\mathbb{T}^3}|U^{\epsilon}_j(r)|^{p-2}|\nabla U^{\epsilon}_j(r)|^2dxdr \right]\\
    &\leq C_{\delta}\mathbb{E}\int_0^T\left\|U^\epsilon\right\|_{p}^{p} d r + K\mathbb{E}\int_0^T\int_{\mathbb{T}^3}|U^{\epsilon}_j(r)|^{p-2}|\nabla U^{\epsilon}_j(r)|^2dxdr  \\ 
    &\quad + C\mathbb{E}\left[\|U^{\epsilon}_0\|_p^p + \int_0^T\left\|f^{\epsilon}(r)\right\|_{-1, q}^p d r + \int_0^T\left\|(GU^{\epsilon})_j(r)\right\|_{-1, p}^p d r + \int_0^T\|g^{\epsilon}(r)\|_{\mathbb{L}^p}^pdr\right].
    \end{split}
\end{align}
Here 
\begin{align*}
K = \delta + 2p(p-1)N_{b, 0}+p^2C_{BDG}^2N_{b,0}.
\end{align*}
For $j=4$, observe that $(GU^{\epsilon})_4\equiv 0$ and $ (\mathbf{Q}\left(b\cdot\nabla U^{\epsilon}\right))_4 \equiv0$. Therefore from \eqref{e.L021404}, \eqref{e.L022102}-\eqref{e.L121803}, and  \eqref{e.L022105}-\eqref{e.L022107}, we obtain
\begin{align}\label{e.L022201}
\begin{split}
    &\mathbb{E}\left[\frac12\sup_{t\in[0, T]}\|U^{\epsilon}_4(t)\|_{p}^p +p(p-1)\int_0^T\int_{\mathbb{T}^3}|U^{\epsilon}_4(r)|^{p-2}|\nabla U^{\epsilon}_4(r)|^2dxdr \right]\\
    &\leq C_\delta \mathbb{E}\int_0^T\left\|U_4^\epsilon\right\|_{p}^{p} d r  + K\mathbb{E}\int_0^T\int_{\mathbb{T}^3}|U^{\epsilon}_4(r)|^{p-2}|\nabla U^{\epsilon}_4(r)|^2dxdr  \\ 
    &\quad + C\mathbb{E}\left[\|U^{\epsilon}_{0,4}\|_p^p + \int_0^T\left\|f^{\epsilon}(r)\right\|_{-1, q}^p d r  + \int_0^T\|g^{\epsilon}(r)\|_{\mathbb{L}^p}^pdr\right].
\end{split}
\end{align}
In view of the condition \eqref{e.L021403} on $b$, we can choose $\delta>0$ small such that $p(p-1)-K>0$. Fix such a $\delta$. Then 
\begin{align*}
\begin{split}
    &\mathbb{E}\left[\sup_{t\in[0, T]}\|U^{\epsilon}_4(t)\|_{p}^p +\int_0^T\int_{\mathbb{T}^3}|U^{\epsilon}_4(r)|^{p-2}|\nabla U^{\epsilon}_4(r)|^2dxdr \right]\\
    &\leq C\mathbb{E}\left[\|U^{\epsilon}_{0,4}\|_p^p + \int_0^T\left\|U_4^\epsilon\right\|_{p}^{p} d r   + \int_0^T\left\|f^{\epsilon}(r)\right\|_{-1, q}^p d r  + \int_0^T\|g^{\epsilon}(r)\|_{\mathbb{L}^p}^pdr\right].
\end{split}
\end{align*}
It follows from Gronwall's inequality that 
\begin{align*}
\begin{split}
    &\mathbb{E}\left[\sup_{t\in[0, T]}\|U^{\epsilon}_4(t)\|_{p}^p +\int_0^T\int_{\mathbb{T}^3}|U^{\epsilon}_4(r)|^{p-2}|\nabla U^{\epsilon}_4(r)|^2dxdr \right]\\
    &\leq C\mathbb{E}\left[\|U^{\epsilon}_{0,4}\|_p^p  + \int_0^T\left\|f^{\epsilon}(r)\right\|_{-1, q}^p d r  + \int_0^T\|g^{\epsilon}(r)\|_{\mathbb{L}^p}^pdr\right].
\end{split}
\end{align*}
This shows that for $j=1,2,3$, by the boundedness of Leray projection,  we have 
\begin{align}\label{e.L021407}
\begin{split}
\int_0^T\left\|(GU^{\epsilon})_j(r)\right\|_{-1, p}^p d r &= \int_0^T\left\|(\mathcal{P}(U^{\epsilon}_4 e_3))_j(r)\right\|_{-1, p}^p d r\\
&\leq C \int_0^T\left\|U^{\epsilon}_4(r) e_3\right\|_{-1, p}^p d r\leq C_T\sup_{t\in[0, T]}\|U^{\epsilon}_4(t)\|_{p}^p\\
&\leq C\mathbb{E}\left[\|U^{\epsilon}_{0,4}\|_p^p  + \int_0^T\left\|f^{\epsilon}(r)\right\|_{-1, q}^p d r  + \int_0^T\|g^{\epsilon}(r)\|_{\mathbb{L}^p}^pdr\right].
\end{split}
\end{align}
Combining \eqref{e.L0221012}, \eqref{e.L022201} and \eqref{e.L021407}, one obtains
\begin{align}\label{e.L022203}
\begin{split}
    &\mathbb{E}\left[\sup_{t\in[0, T]}\|U^{\epsilon}(t)\|_{p}^p +\sum_{j=1}^4\int_0^T\int_{\mathbb{T}^3}|U^{\epsilon}_j(r)|^{p-2}|\nabla U^{\epsilon}_j(r)|^2dxdr \right]\\
    &\leq  C\mathbb{E}\left[\|U^{\epsilon}_0\|_p^p + \int_0^T\left\|U^\epsilon\right\|_{p}^{p} d r   + \int_0^T\left\|f^{\epsilon}(r)\right\|_{-1, q}^p d r  + \int_0^T\|g^{\epsilon}(r)\|_{\mathbb{L}^p}^pdr\right].
    \end{split}
\end{align}
Again using Gronwall's inequality we obtain the desired estimate
\begin{align}\label{e.L022204}
\begin{split}
    &\mathbb{E}\left[\sup_{t\in[0, T]}\|U^{\epsilon}(t)\|_{p}^p +\sum_{j=1}^4\int_0^T\int_{\mathbb{T}^3}|U^{\epsilon}_j(r)|^{p-2}|\nabla U^{\epsilon}_j(r)|^2dxdr \right]\\
    &\leq  C\mathbb{E}\left[\|U^{\epsilon}_0\|_p^p  + \int_0^T\left\|f^{\epsilon}(r)\right\|_{-1, q}^p d r  + \int_0^T\|g^{\epsilon}(r)\|_{\mathbb{L}^p}^pdr\right].
    \end{split}
\end{align}

\vskip0.1in
{\it Step 3}. By the linearity of the equation and the result of the previous step, we have for $U^{\epsilon}-U^{\epsilon'}$, 
\begin{align*}
\begin{split}
    \mathbb{E}\Bigg[\sup_{t\in[0, T]}\|U^{\epsilon}(t)-&U^{\epsilon'}(t)\|_{p}^p \Bigg]
  \leq C\mathbb{E}\left[\|U^{\epsilon}_0-U^{\epsilon'}_0\|_p^p + \int_0^T\left\|f^{\epsilon}(r)-f^{\epsilon'}(r)\right\|_{-1, q}^p d r + \int_0^T\|g^{\epsilon}(r)-g^{\epsilon'}(r)\|_{\mathbb{L}^p}^pdr\right].
\end{split}
\end{align*}
By the fundamental properties of mollifiers, we see that the right hand side of the this inequality goes to zero as $\epsilon,\epsilon'\to0$. Therefore, there exists an element $U\in L^p(\Omega, C([0, T], L^p))$ and a subsequence $U^{\epsilon_n}\to U$ in $L^p(\Omega, L^{\infty}([0, T], L^p))$, where $\epsilon_n\to0$ as $n\to\infty$. Using integration by parts as in \cite{KXZ}, one can show that $U$ is a pathwise solution of \eqref{e.L021401}. Indeed, one has 
\begin{align*}
(U^{\epsilon_n}, \phi) = (U_0^{\epsilon_n}, \phi) &+ \int_0^t \Big((U^{\epsilon_n},A \phi) + (GU^{\epsilon_n} + f^{\epsilon_n}, \phi)\Big)dr 
\\
&+ \int_0^t \Big(-(\mathbf{P}\left(b \otimes U^{\epsilon_n} \right) , \nabla \phi) + (g^{\epsilon_n}, \phi) \Big)d\mathbb{W}_r, \quad (t,\omega)-\mathrm{a.e.},
\end{align*}
for all $\phi\in C^{\infty}(\mathbb{T}^3;\mathbb{R}^4)$ and $n\geq 1$. Then Hölder's inequality and dominated convergence theorem imply that as $n\to\infty$, 
\begin{align*}
(U^{\epsilon_n}, \phi)  - (U_0^{\epsilon_n}, \phi)&\to (U, \phi) - (U_0, \phi),\\
\int_0^t\Big(( U^{\epsilon_n}, A \phi) + (GU^{\epsilon_n} + f^{\epsilon_n}, \phi)\Big)dr&\to\int_0^t\Big((U, A\phi) +( GU + f, \phi)\Big)dr.
\end{align*}
By the BDG inequality and Minkowski's inequality, we have  
\begin{align*}
&\mathbb{E}\sup_{t\in[0, T]}\left|\int_0^t\Big(-\mathbf{P}(b\otimes(U^{\epsilon_n}-U)) , \nabla\phi)+ (g^{\epsilon_n}-g, \phi)\Big)d\mathbb{W}_r\right|\\
&\leq C\mathbb{E}\left(\int_0^T\left\|(\mathbf{P}(b\otimes(U^{\epsilon_n}-U)), \nabla\phi)\right\|_{l^2}^2dr\right)^{1/2}+C\mathbb{E}\left(\int_0^T\left\|( g^{\epsilon_n}-g,\phi)\right\|_{l^2}^2dr\right)^{1/2}\\
&\leq C_b\|\nabla\phi\|_{L^2}\|\mathbb{E}\sup_{t\in[0, T]}\|U^{\epsilon_n}(t)-U(t)\|_{p}+ C\|\phi\|_{L^2}\mathbb{E}\int_0^T\|g^{\epsilon_n}(t)-g(t)\|_{\mathbb{L}^p}^pdt,
\end{align*}
which converge to $0$ as $n\to\infty$. Hence there is a subsequence $\epsilon_{n_j}$ such that 
\[\int_0^t\Big(-\mathbf{P}(b\otimes U^{\epsilon_{n_j}}),\nabla\phi) +( g^{\epsilon_{n_j}}, \phi)\Big)d\mathbb{W}_r\to \int_0^t\Big(-\mathbf{P}(b\otimes U),\nabla\phi)+ (g, \phi)\Big)d\mathbb{W}_r, \quad (t,\omega)-\mathrm{a.e.}.\]
Hence $U$ is a pathwise solution of \eqref{e.L021401}. Lemma 4.4 in \cite{KXZ} with \eqref{e.L022204} also imply that \eqref{e.L021402} is true. 

To show the uniqueness,  we let $U_1, U_2$ be two pathwise solutions to \eqref{e.L021401}, then their difference $V=U_1-U_2$ solves 
\begin{align*}
\begin{split}
&dV = (A V + GV)dt + \mathbf{P}(b\cdot\nabla V) d\mathbb{W}_t, \\
&U(0) = 0. 
\end{split}
\end{align*}
Then Ito's formula and estimates as above give 
\begin{align*}
\mathbb{E}\sup _{0 \leq t \leq T}\|V(t, \cdot)\|_{p}^{p} =0. 
\end{align*}
So $V\equiv0$ almost surely. 
\end{proof}

\section{Local existence and uniqueness}\label{s.w02271}
In this section we will establish the existence and uniqueness of maximal pathwise solution for Boussinesq system \eqref{e.L10261}.
\subsection{Truncated system}
We first consider the following truncated system 
\begin{equation}\label{e.Q02181}
\begin{split}
&d U-\Delta U d t
=
\varphi\left(\left\|U\right\|_{p}\right)^2  B(U) dt +G(U) d t 
+\varphi\left(\left\|U\right\|_{p}\right)^2 \sigma(U) d \mathbb{W}_{t}+\mathbf{P}(b\cdot\nabla U) d \mathbb{W}_{t}, 
 \\&\nabla \cdot u=0, 
 \\& U(0)=U_{0},
\end{split}
\end{equation}
on $[0, \infty) \times \mathbb{T}^{3}$ with $\nabla \cdot u_{0}=0$ and $\int_{\mathbb{T}^{3}} U_{0} d x=0$ a.s..
Here for some fixed $\delta_0>0$ we denote by $\varphi:[0,\infty)\to [0,1]$ a decreasing smooth function such that $\varphi\equiv 1$ on $[0,\frac{\delta_0}{2}]$ and $\varphi\equiv 0$ on $[\delta_0,\infty).$ In addition, we have the Lipschitz continuity for $\varphi$:
\[
|\varphi(x_1)-\varphi(x_2)|\leq \frac{C}{\delta_0} |x_1-x_2|.
\]

The following theorem concerns the existence and uniqueness of solution to system \eqref{e.Q02181}.
\begin{theorem}
\label{t.02051} 
Let $p>5$ and $U_{0} \in L^{p}\left(\Omega ; L^{p}\right)$. For every $T>0$, there exists a unique pathwise solution $u \in L^{p}\left(\Omega ; C\left([0, T], L^{p}\right)\right)$ to \eqref{e.Q02181} such that
\begin{equation}
\mathbb{E}\left[\sup _{0 \leq s \leq T}\|U(s, \cdot)\|_{p}^{p}+\sum_{j} \int_{0}^{T} \int_{\mathbb{T}^{3}}\left|\nabla\left(\left|U_{j}(s, x)\right|^{p / 2}\right)\right|^{2} d x d s\right] \leq C \mathbb{E}\left[\left\|U_{0}\right\|_{p}^{p}\right]+C_{T}.
\end{equation}
\end{theorem}

In order to solve system \eqref{e.Q02181}, we consider the first iteration procedure
\begin{equation}\label{e.Q02101}
\begin{split}
d U^{(n)}-\Delta U^{(n)} d t
&=
\varphi\left(\left\|U^{(n)}\right\|_{p}\right) \varphi\left(\left\|U^{(n-1)}\right\|_{p}\right)  B(U^{(n-1)}) dt +G(U^{(n)}) d t 
\\&\quad+\varphi\left(\left\|U^{(n)}\right\|_{p}\right) \varphi\left(\left\|U^{(n-1)}\right\|_{p}\right) \sigma(U^{(n-1)}) d \mathbb{W}_{t} 
+
\mathbf{P}(b\cdot\nabla U^{(n)})d \mathbb{W}_{t} 
 \\&\nabla \cdot u^{(n)}=0, 
 \\& U^{(n)}(0)=U_{0} ,
\end{split}
\end{equation}
where $U^{(0)}$ is the pathwise solution to
\begin{equation}\label{e.Q02102}
\begin{split}
& d U^{(0)}-\Delta U^{(0)} d t    = G(U^{(0)}) dt + \mathbf{P}( b\cdot\nabla U^{(0)}) d \mathbb{W}_{t} , \\
& \nabla \cdot u^{(0)}=0, \\
& U^{(0)}(0)=U_{0} .
\end{split}
\end{equation}
By Threorem \ref{t.L021401}, we obtain an unique solution $U^{(0)} \in L^{p}(\Omega ; C([0, T], L^{p}))$ and
\begin{align*}
\begin{split}
    \mathbb{E}\left[\sup _{0 \leq t \leq T}\|U^{(0)}(t, \cdot)\|_{p}^{p}+\sum_{j=1}^{4} \int_{0}^{T} \int_{\mathbb{T}^{3}}\left|\nabla\left(\left|U^{(0)}_{j}(t, x)\right|^{p / 2}\right)\right|^{2} d x d t\right] 
    \leq  C\mathbb{E}\left[\left\|U_{0}\right\|_{p}^{p}\right].
\end{split}
\end{align*}
Next, one needs to establish the existence of unique solution to system \eqref{e.Q02101} for each $n\in\mathbb N$. For this purpose, we consider
\begin{align}\label{e.Q02103}
    &dU - \Delta U dt =  \varphi\left(\|U\|_p \right)\varphi\left(\|V\|_p \right) B(V) dt + G(U) dt + \varphi\left(\|U\|_p \right)\varphi\left(\|V\|_p \right)\sigma(V) d\mathbb{W}_t + \mathbf{P}(b\cdot\nabla U) d \mathbb{W}_{t} \nonumber  
    \\
    &\nabla\cdot u = 0\nonumber
    \\
    & U(0) = U_0  ,
\end{align}
where $V=(v,\theta)$ satisfying $\nabla\cdot v= 0$ and
\begin{align}\label{e.Q02104}
    \mathbb{E}\left[\sup _{0 \leq t \leq T}\left\|V(t, \cdot)\right\|_{p}^{p}+\sum_{j=1}^4 \int_{0}^{T} \int_{\mathbb{T}^{3}}\left|\nabla\left(\left|V_{j}(t, x)\right|^{p / 2}\right)\right|^{2} d x d t\right] \leq C \mathbb{E}\left[\left\|U_{0}\right\|_{p}^{p}\right]+C_{T}.
\end{align}
To the end, we employ the second iteration process
\begin{align}\label{e.Q02105}
    &dU^{(n)} - \Delta U^{(n)}dt = \varphi\left(\|U^{(n-1)}\|_p\right)\varphi\left(\|V\|_p\right) B(V) dt + G(U^{(n)}) dt \nonumber
    \\
    &\hspace{4cm}+ \varphi\left(\|U^{(n-1)}\|_p\right)\varphi\left(\|V\|_p\right) \sigma(V) d\mathbb{W}_t  +\mathbf{P}( b\cdot\nabla U^{(n)} ) d \mathbb{W}_{t}, \nonumber 
    \\
    &\nabla\cdot u^n = 0,\nonumber
    \\
    & U^n(0) = U_0  ,
\end{align}
in order to solve system \eqref{e.Q02103}. Here $U^{(0)}$ is the solution to \eqref{e.Q02102}, but for $n\geq 1$ the $U^{(n)}$ in \eqref{e.Q02105} is different from the one in \eqref{e.Q02101}.


The following lemma concerns the induction step in order to establish the existence of unique solution to system \eqref{e.Q02105} for each $n$.
 \begin{lemma}
 \label{l.w02053}
 Let $p>5, n \in \mathbb{N}$, and $T>0$. Suppose $U_{0} \in L^{p}\left(\Omega ; L^{p}\right)$ and assume that for each $n \in\{1,2, \ldots, k-1\}$, there exists a unique solution $U^{(n)} \in$ $L^{p}\left(\Omega ; C\left([0, T], L^{p}\right)\right)$ to the initial value problem \eqref{e.Q02105}, where $V$ and $U^{(n)}$ satisfy \eqref{e.Q02104}. Then for $n=k$, the initial value problem \eqref{e.Q02105} also has a unique solution $U^{(k)} \in L^{p}\left(\Omega ; C\left([0, T], L^{p}\right)\right)$, and moreover,
\begin{equation}\label{e.q02113}
\mathbb{E}\left[\sup _{0 \leq t \leq T}\left\|U^{(k)}(t, \cdot)\right\|_{p}^{p}+\sum_{j=1}^4 \int_{0}^{T} \int_{\mathbb{T}^{3}}\left|\nabla\left(\left|U_{j}^{(k)}(t, x)\right|^{p / 2}\right)\right|^{2} d x d t\right] \leq C \mathbb{E}\left[\left\|U_{0}\right\|_{p}^{p}\right]+C_{T}.
\end{equation}
 \end{lemma}
 
Assuming Lemma \ref{l.w02053} holds, one can achieve the existence of unique solution to \eqref{e.Q02105} for each $n$ from the existence of $U^{(0)}$ solving \eqref{e.Q02102} and mathematical induction. The next lemma shows the existence of unique solution to \eqref{e.Q02103}.

\begin{lemma}
\label{l.w02054}
Let $p>5$ and suppose that $U_{0} \in L^{p}\left(\Omega ; L^{p}\right)$. Then there exists a time $t>0$ small enough such that the initial value problem \eqref{e.Q02103}, where V satisfies \eqref{e.Q02104}, has a unique pathwise solution $U \in L^{p}\left(\Omega ; C\left([0, t], L^{p}\right)\right)$, which satisfies
\begin{equation}\label{e.Q02191}
\mathbb{E}\left[\sup _{0 \leq s \leq t}\|U(s, \cdot)\|_{p}^{p}+\sum_{j=1}^4 \int_{0}^{t} \int_{\mathbb{T}^{3}}\left|\nabla\left(\left|U_{j}(s, x)\right|^{p / 2}\right)\right|^{2} d x d s\right] \leq C \mathbb{E}\left[\left\|U_{0}\right\|_{p}^{p}\right]+C_{t}
\end{equation}
\end{lemma}

We postpone the proof of Lemma \ref{l.w02053} and \ref{l.w02054} to Section \ref{sec:4.3}.
Assuming Lemma \ref{l.w02054} holds, we are ready to prove Theorem \ref{t.02051}.

\begin{proof}[Proof of Theorem \ref{t.02051}]
Consider the system \eqref{e.Q02101}. Thanks to Lemma \ref{l.w02054} and by induction we know that for each $n\in\mathbb N$ there exists an unique solution $U^{(n)} \in L^P(\Omega; C([0,T], L^p))$ to system \eqref{e.Q02101} with $T>0$ sufficiently small, and satisfies
\begin{equation*}
\mathbb{E}\left[\sup _{0 \leq t \leq T}\left\|U^{(n)}(t, \cdot)\right\|_{p}^{p}+\sum_{j=1}^4 \int_{0}^{T} \int_{\mathbb{T}^{3}}\left|\nabla\left(\left|U_{j}^{(n)}(t, x)\right|^{p / 2}\right)\right|^{2} d x d t\right] \leq C \mathbb{E}\left[\left\|U_{0}\right\|_{p}^{p}\right]+C_{T}.
\end{equation*}

Next, we consider the difference  $V^{(n)}=U^{(n+1)}-U^{(n)}$. For each $n\in\mathbb N$ denote by 
\begin{align*}
   \varphi^{(n)} = \varphi\left(\|U^{(n)}\|_p\right).
\end{align*}
Thanks to the linearity of $G$ and the transport noise, we have
\begin{equation}\label{e.Q02241}
\begin{split}
d V^{(n)}-\Delta V^{(n)} d t
&=
\varphi^{(n+1)} \varphi^{(n)}  B(U^{(n)})  d t 
-
\varphi^{(n)} \varphi^{(n-1)}  B(U^{(n-1)})  d t + G(V^{(n)}) dt
\\&\quad+
(\varphi^{(n+1)} \varphi^{(n)} \sigma(U^{(n)}) 
-
\varphi^{(n)} \varphi^{(n-1)} \sigma(U^{(n-1)})) d \mathbb{W}_{t}
+
\mathbf{P}(b\cdot\nabla V^{(n)}) d\mathbb{W}_{t}, 
\\
\nabla\cdot v^{(n)} &=0,
\\
V^{(n)}(0) &= 0 \quad a.s..
\end{split}
\end{equation}
The first equation can be rewritten as 
\begin{align*}
    d V^{(n)}-\Delta V^{(n)} d t
&=
\sum_{i=1}^3 \partial_i f_i  d t + G(V^{(n)}) dt
+
g d \mathbb{W}_{t}
+
\mathbf{P}(b\cdot\nabla V^{(n)})d\mathbb{W}_{t},
\end{align*}
where 
\begin{equation}
\begin{split}
f_{i}= & -\varphi^{(n+1)} \varphi^{(n)}\left(\mathbf{P}\left(u_{i}^{(n)} U^{(n)}\right)\right)+\varphi^{(n)} \varphi^{(n-1)}\left(\mathbf{P}\left(u_{i}^{(n-1)} U^{(n-1)}\right)\right) \\
= & -\varphi^{(n)}\left(\varphi^{(n+1)}-\varphi^{(n)}\right)\left(\mathbf{P}\left(u_{i}^{(n)} U^{(n)}\right)\right)-\varphi^{(n)}\left(\varphi^{(n)}-\varphi^{(n-1)}\right)\left(\mathbf{P}\left(u_{i}^{(n)} U^{(n)}\right)\right) \\
& -\varphi^{(n)} \varphi^{(n-1)}\left(\mathbf{P}\left(v_{i}^{(n-1)} U^{(n)}\right)\right)-\varphi^{(n)} \varphi^{(n-1)}\left(\mathbf{P}\left(u_{i}^{(n-1)} V^{(n-1)}\right)\right) \\
= & f_{i}^{(1)}+f_{i}^{(2)}+f_{i}^{(3)}+f_{i}^{(4)},
\end{split}
\end{equation}
and
\begin{equation}
\begin{split}
g= & \varphi^{(n+1)} \varphi^{(n)} \sigma\left(U^{(n)}\right)-\varphi^{(n)} \varphi^{(n-1)} \sigma\left(U^{(n-1)}\right) \\
= & \varphi^{(n)}\left(\varphi^{(n+1)}-\varphi^{(n)}\right) \sigma\left(U^{(n)}\right)+\varphi^{(n)}\left(\varphi^{(n)}-\varphi^{(n-1)}\right) \sigma\left(U^{(n)}\right) \\
& +\varphi^{(n)} \varphi^{(n-1)}\left(\sigma\left(U^{(n)}\right)-\sigma\left(U^{(n-1)}\right)\right) \\
= & g^{(1)}+g^{(2)}+g^{(3)} .
\end{split}
\end{equation}
As $\mathbf{P}$ is a bounded operator, we can now use the estimates of (5.19) and (5.20) in \cite{KXZ} to get
\begin{equation}\label{e.Q02243}
    \begin{split}
        &\mathbb E\left[ \int_0^t \|f\|_q^p ds \right] \leq C t \mathbb E\left[ \sup\limits_{s\in[0,t]} \|V^{(n-1)}\|_p^p \right] + C t \mathbb E\left[ \sup\limits_{s\in[0,t]} \|V^{(n)}\|_p^p \right] ,
    \\
    &\mathbb E\left[ \int_0^t \|g\|_{\mathbb L^p}^p ds \right] \leq C t \mathbb E\left[ \sup\limits_{s\in[0,t]} \|V^{(n-1)}\|_p^p \right] + C t \mathbb E\left[ \sup\limits_{s\in[0,t]} \|V^{(n)}\|_p^p \right].
    \end{split}
\end{equation}
Thanks to Theorem \ref{t.L021401} and the fact that $V^{(n)}(0) = 0$, we know 
\begin{align}\label{e.Q02244}
    \mathbb{E}\left[\sup _{s\in[0,t]}\|V^{(n)}(s, \cdot)\|_{p}^{p}\right] 
   \leq  C t \mathbb E\left[ \sup\limits_{s\in[0,t]} \|V^{(n-1)}\|_p^p \right] + C t \mathbb E\left[ \sup\limits_{s\in[0,t]} \|V^{(n)}\|_p^p \right].
\end{align}
In particular, by choosing $t=t^*$ small enough we obtain exponential convergence rate, and there exists a fixed point $U \in L^p(\Omega; C([0,t^*],L^p))$ of system \eqref{e.Q02101}. As $U^{(n)}$ is the solution to system \eqref{e.Q02101} for each $n$, we have
\begin{align*}
    \left(U^{(n)}(s), \phi\right)= & \left(U_0, \phi\right)+\int_0^s\left(U^{(n)}(r), A \phi\right) d r  + \int_0^s\left( G(U^{(n)}(r)),\phi\right) d r\\
& +\sum_{i=1}^3 \int_0^s\left(\varphi^{(n)} \varphi^{(n-1)} \mathbf{P}\left(u_i^{(n-1)} U^{(n-1)}\right), \partial_i \phi\right) d r \\
& +\int_0^s\left(\varphi^{(n)} \varphi^{(n-1)} \sigma\left(U^{(n-1)}\right), \phi\right) d \mathbb{W}_r - \sum_{i=1}^3 \int_0^s\left(\mathbf{P}\left( b_i U^{(n)}\right), \partial_i\phi\right) d \mathbb{W}_r,
\end{align*}
for a.a. $(s,\omega)\in (0,t^*)\times \Omega$ and all $\phi \in C^{\infty}\left(\mathbb{T}^3;\mathbb R^4\right)$. By the exponential convergence rate and thanks to Lemma 5.2 and Remark 5.3 in \cite{KXZ}, one has 
\[
\|U^{(n)}\|_p \to \|U\|_p \quad \text{ for a.a. } (s,\omega)\in (0,t^*)\times \Omega.
\]
Therefore, 
$
\left(U^{(n)}(s), \phi\right) \to \left(U(s), \phi\right), 
$
$\varphi^{(n)}$ and $\varphi^{(n-1)} \to \varphi(\|U(s)\|_p):=\varphi$ for a.a. $(s,\omega)\in (0,t^*)\times \Omega$. Thanks to the dominant convergence theorem, one obtains
\begin{align*}
    &\int_0^s\left(U^{(n)}(r), A \phi\right) d r +  \int_0^s\left( G(U^{(n)}(r)),\phi\right) d r
    +\sum_{i=1}^3 \int_0^s\left( \varphi^{(n)} \varphi^{(n-1)}\mathbf{P}\left(u_i^{(n-1)} U^{(n-1)}\right), \partial_i \phi\right) d r 
    \\
    &\hspace{1cm}\to \int_0^s\left(U(r), A \phi\right) d r +  \int_0^s\left( G(U(r)),\phi\right) d r
    +\sum_{i=1}^3 \int_0^s\left(\varphi^2\mathbf{P}\left(u_i U\right), \partial_i \phi\right) d r
\end{align*}
for a.s. $(s,\omega)\in (0,t^*)\times \Omega$. Next, by BDG inequality and thanks to the property of $\sigma$, we have
\begin{align*}
&\mathbb{E}\left[\sup _{s \in\left[0, t^*\right)}\left|\int_0^s\left(\varphi^{(n)} \varphi^{(n-1)} \sigma\left(U^{(n-1)}\right)-\varphi^2 \sigma(U), \phi\right) d \mathbb{W}_r\right|\right] \\
\leq &\mathbb{E}\left[\left(\int_0^{t^*}\left\|\left(\varphi^{(n)} \varphi^{(n-1)} \sigma\left(U^{(n-1)}\right)-\varphi^2 \sigma(U), \phi\right)\right\|_{l^2}^2 d r\right)^{1 / 2}\right] \to 0 ,
\end{align*}
and
\begin{align*}
   &\mathbb{E}\left[\sup _{s \in\left[0, t^*\right)}\left|\int_0^s\sum_{i=1}^3  \left( \mathbf{P}\left( b_i U^{(n)}\right) - \mathbf{P}\left( b_i U\right), \partial_i\phi\right) d \mathbb{W}_r\right|\right] \\
\leq & C_{BDG} \sum_{i=1}^3\mathbb{E}\left[\left(\int_0^{t^*}\left\|\left( \mathbf{P}\left( b_i U^{(n)}\right) - \mathbf{P}\left( b_i U\right), \partial_i\phi\right)\right\|_{l^2}^2 d r\right)^{1 / 2}\right] \to 0 .
\end{align*}
Thanks to Lemma 5.2 and Remark 5.3 in \cite{KXZ}, one has
\begin{align*}
   &\int_0^s\left(\varphi^{(n)} \varphi^{(n-1)} \sigma\left(U^{(n-1)}\right) ,\phi\right)  d \mathbb{W}_r \to \int_0^s \left(\varphi^2 \sigma(U), \phi\right) d \mathbb{W}_r ,
   \\
   &\sum_{i=1}^3 \int_0^s\left( \mathbf{P}\left( b_i U^{(n)}\right), \partial_i\phi\right) d \mathbb{W}_r \to \sum_{i=1}^3 \int_0^s\left( \mathbf{P}\left( b_i U\right), \partial_i\phi\right) d \mathbb{W}_r.
\end{align*}
Combining the discussion above, by taking $n\to \infty$ we get
\begin{align*}
    \left(U(s), \phi\right)= & \left(U_0, \phi\right)+\int_0^s\left(U(r), A \phi\right) d r  + \int_0^s\left( G(U(r)),\phi\right) d r +\sum_{i=1}^3 \int_0^s\left(\varphi^2 \mathbf{P}\left(u_i U\right), \partial_i \phi\right) d r \\
& +\int_0^s\left(\varphi^2\sigma\left(U\right), \phi\right) d \mathbb{W}_r - \sum_{i=1}^3 \int_0^s\left( \mathbf{P}\left( b_i U\right), \partial_i\phi\right) d \mathbb{W}_r.
\end{align*}
Therefore we obtain the existence of a solution $U \in L^p(\Omega; C([0,t^*],L^p))$ to system \eqref{e.Q02181}.

Next, for the uniqueness we consider two pathwise solutions $U,V \in L^p (\Omega; C([0,t^*], L^p))$ to system \eqref{e.Q02181}. By denoting $W=U-V$ we have
\begin{align}\label{e.Q02242}
& d W-\Delta W d t=\left(\varphi_U^2 B(U)-\varphi_V^2 B(V) + G(W)\right) d t+\left(\varphi_U^2 \sigma(U)-\varphi_V^2 \sigma(V) + \mathbf{P}(b\cdot \nabla W)\right) d \mathbb{W}_t, \nonumber\\
& \nabla \cdot (W_1,W_2,W_3)=0, \nonumber\\
& W(0)=0, \quad \text { a.s.. }
\end{align}
Here $\varphi_U = \varphi(\|U\|_p)$ and $\varphi_V = \varphi(\|V\|_p)$. Similar as the treatment for \eqref{e.Q02241}, we can rewrite the first equation of \eqref{e.Q02242} as 
\begin{align*}
    dW-\Delta W d t
&=
\sum_{i=1}^3 \partial_i f_i  d t + G(W) dt
+
g d \mathbb{W}_{t}
+
\mathbf{P}(b\cdot\nabla W)d\mathbb{W}_{t},
\end{align*}
where
\begin{align*}
f_{i}= & -\varphi_U^2\left(\mathbf{P}\left(u_i U\right)\right)+\varphi_V^2\left(\mathbf{P}\left(v_i V\right)\right) \\
= & -\varphi_U\left(\varphi_U-\varphi_V\right)\left(\mathbf{P}\left(u_i U\right)\right)-\varphi_U \varphi_V\left(\mathbf{P}\left(w_i U\right)\right) \\
& -\varphi_U \varphi_V\left(\mathbf{P}\left(v_i W\right)\right)-\varphi_V\left(\varphi_U-\varphi_V\right)\left(\mathbf{P}\left(v_i V\right)\right)
\end{align*}
and
\begin{align*}
g & =\varphi_U^2 \sigma(U)-\varphi_V^2 \sigma(V) \\
& =\varphi_U\left(\varphi_U-\varphi_V\right) \sigma(U)+\varphi_V\left(\varphi_U-\varphi_V\right) \sigma(V)+\varphi_U \varphi_V\left(\sigma(U)-\sigma(V)\right) .
\end{align*}
Similar to \eqref{e.Q02243} we obtain
\begin{equation}\label{e.Q02245}
    \begin{split}
        \mathbb E\left[ \int_0^{t^*} \|f\|_q^p ds \right] \leq C t^* \mathbb E\left[ \sup\limits_{s\in[0,t^*]} \|W\|_p^p \right]  , \quad
\mathbb E\left[ \int_0^{t^*} \|g\|_{\mathbb L^p}^p ds \right] \leq C t^* \mathbb E\left[ \sup\limits_{s\in[0,t^*]} \|W\|_p^p \right].
    \end{split}
\end{equation}
Therefore thanks to Theorem \ref{t.L021401} we obtain
\begin{align}\label{e.Q02246}
    \mathbb E\left[ \sup\limits_{s\in[0,t^*]} \|W\|_p^p\right] \leq Ct^* \mathbb E\left[ \sup\limits_{s\in[0,t^*]} \|W\|_p^p\right].
\end{align}
When $t^*$ is small enough we obtain that $W=0$. Thus the solution is unique.

Finally, note that the constant $C$ appearing in \eqref{e.Q02244} and \eqref{e.Q02245} does not depend on the initial data, and thus $t^*$ is independent of the initial data. For arbitrary $T>0$, one can consider $N$ large enough such that $\frac TN < t^*$. Then by establishing the existence and uniqueness of pathwise solution inductively on $[\frac{i}{N}T,\frac{i+1}{N}T]$ for $i\in\{0,1,2,...N-1\}$, we obtain a unique pathwise solution on $[0,T]$. As $T$ is arbitrary, the solution exists globally in time.
\end{proof}

\subsection{Proof of Theorem \ref{t.w10101}}
Now by applying a family of suitable stopping times we are able to prove the existence and uniqueness of maximal pathwise solution to system \eqref{e.L10261}, which gives Theorem \ref{t.w10101}.

\begin{proof}[Proof of Theorem \ref{t.w10101}]
    For $n\in\mathbb N$, denote by $U^{(n)}$ the solution of the truncated system \eqref{e.Q02181} with $\delta_0=n$. Also, introduce the corresponding stopping times
$$
\tau_n(\omega)= \begin{cases}\inf \left\{t>0:\left\|U^{(n)}(t, \omega)\right\|_p \geq n / 2\right\}, & \text { if }\left\|U^{(n)}(0, \omega)\right\|_p<n / 2, \\ 0, & \text { if }\left\|U^{(n)}(0, \omega)\right\|_p \geq n / 2\end{cases}
$$
By uniqueness, the sequence is non-decreasing a.s. and $U^{(m)}=U^{(n)}$ on $\left[0, \tau_m \wedge \tau_n\right]$. Let $\tau=\lim _n \tau_n \wedge T$. Then, $\mathbb{P}(\tau>0)=1$. Also, for any integer $n \in \mathbb{N}$, define $U=U^{(n)}$ on $\left[0, \tau_n \wedge T\right]$. It is easy to check that $(U, \tau)$ satisfies all the required properties.
\end{proof}

\subsection{Proof of Lemma \ref{l.w02053} and \ref{l.w02054}}\label{sec:4.3}

\begin{proof}[Proof of Lemma \ref{l.w02053}]
Let $n=k$, and denote by
\[
\varphi^{(n-1)} = \varphi\left(\|U^{(n-1)}\|_p\right),\quad \varphi_V = \varphi\left(\|V\|_p\right),
\]
then we can rewrite the first equation in \eqref{e.Q02105} as
\begin{align}\label{e.Q02106}
d U^{(n)}-\Delta U^{(n)} d t  =\varphi^{(n-1)} \varphi_V B(V) d t  +G(U^{(n)}) dt 
+  \varphi^{(n-1)} \varphi_V \sigma(V) d \mathbb{W}_t + \mathbf{P} (b\cdot\nabla U^{(n)})d \mathbb{W}_{t}.
\end{align}

For the nonlinear term, we rewrite 
\[
\varphi^{(n-1)} \varphi_V B(V) dt = -\varphi^{(n-1)} \varphi_V \mathbf{P}(v\cdot \nabla V) dt  =-\sum_{i=1}^3\varphi^{(n-1)} \varphi_V \partial_i\left(\mathbf{P}\left(V_i V\right)\right) d t .
\]
In order to apply Theorem \ref{t.L021401} we consider $p,q,r,l$ satisfying
\begin{equation}\label{e.Q02111}
    \frac{3p}{p+1} <q \leq p,\quad \frac{1}{r}+\frac{1}{l}=\frac{1}{q}, \text{ and } l\leq p,\quad r\leq p .
\end{equation}
Thanks to \eqref{e.Q02104} and the boundedness of $\mathbf{P}$, we have
\begin{align*}
     &C \mathbb{E} \int_0^T \|\varphi^{(n-1)} \varphi_V B(V)\|_{-1,q}^p dt \leq C  \sum_{i=1}^3\mathbb{E}  \left[\int_0^T\left\|\varphi^{(n-1)} \varphi_V V_i V\right\|_q^p d s\right] 
     \\
     \leq &C \mathbb{E}\left[\int_0^T \varphi^{(n-1)} \varphi_V\left\|V\right\|_r^p\|V\|_l^p d s\right]  \leq C \mathbb{E}\left[\int_0^T \varphi^{(n-1)} \varphi_V\left\|V\right\|_p^{2p} d s\right] \leq C_{T}.
\end{align*}
As $V$ satisfies \eqref{e.Q02104}, we have $\varphi^{(n-1)} \varphi_VB(V)\in L^p(\Omega\times [0,T], W^{-1,p})$. Here from \eqref{e.Q02111} we know $\frac1r+\frac1l\geq \frac2p$, and thus $\frac2p<\frac{p+1}{3p}$, which gives $p>5$.

Next, by the sub-linear growth of $\sigma(V)$ and the property of $V$, one can easily verify that $\varphi^{(n-1)} \varphi_V \sigma(V)\in L^p(\Omega\times [0,T], \mathbb{L}^p)$. Thus by applying Theorem \ref{t.L021401} we know unique solution $U^{(k)}\in L^p(\Omega;C([0,T],L^p))$ and the bound \eqref{e.q02113} holds.
\end{proof}

\begin{proof}[Proof of Lemma \ref{l.w02054}]
For $U^{(n+1)} = (u^{(n+1)}, \rho^{(n+1)})$ and $U^{(n)} = (u^{(n)}, \rho^{(n)})$ solving system \eqref{e.Q02105}, we consider the difference $Z^{(n)}=U^{(n+1)}-U^{(n)}$, where $Z^{(n)}=(z^{(n)},\kappa^{(n)})$, $z^{(n)}=u^{(n+1)}-u^{(n)}$, $\kappa^{(n)}=\rho^{(n+1)}-\rho^{(n)}$. By direct calculation one gets
\begin{align*}
    d Z^{(n)}-\Delta Z^{(n)} d t=&\sum_{i=1}^{3} (\varphi^{(n)}-\varphi^{(n-1)})\varphi_V \partial_{i} (\mathbf{P}(V_i V)) d t + G(Z^{(n)})  dt
    \\
    & + (\varphi^{(n)}-\varphi^{(n-1)}) \varphi_V \sigma(V) d\mathbb{W}_{t} + \mathbf{P}( b\cdot\nabla Z^{(n)}) d\mathbb{W}_{t},
\end{align*}
where $\nabla\cdot z^{(n)}=0$ and $Z^{(n)}(0) = 0.$
Thanks to the Lipschitz property of $\varphi$, we have
\[
|\varphi^{(n)}-\varphi^{(n-1)}| \leq \frac{C}{\delta_0}\left| \|U^{(n)}\|_p - \|U^{(n-1)}\|_p \right|\leq \frac{C}{\delta_0}\|U^{(n)}-U^{(n-1)}\|_p = \frac{C}{\delta_0} \|Z^{(n-1)}\|_p.
\]
Then by the boundedness of $\mathbf{P}$ we can estimate
\begin{align*}
    \mathbb E\left[\int_0^t \|\sum_{i=1}^{3}  (\varphi^{(n)}-\varphi^{(n-1)})\varphi_V \mathbf{P}(V_i V) \|_{q}^p ds \right] \leq \frac{C}{\delta_0^p}  \mathbb E\left[\int_0^t \|Z^{(n-1)}(s)\|_{p}^p ds \right] \leq Ct \mathbb E\left[\sup\limits_{s\in[0,t]}\|Z^{(n-1)}(s)\|_{p}^p \right],
\end{align*}
and thanks to the property of $\sigma$ and $V$,
\begin{align*}
    \mathbb E \left[\int_0^t \|(\varphi^{(n)}-\varphi^{(n-1)}) \varphi_V \sigma(V) \|_{\mathbb L^p}^p ds \right] \leq C  \mathbb E\left[\int_0^t \|Z^{(n-1)}(s)\|_{p}^p ds \right] \leq Ct \mathbb E\left[\sup\limits_{s\in[0,t]}\|Z^{(n-1)}(s)\|_{p}^p \right].
\end{align*}
By applying Theorem \ref{t.L021401} we conclude that
\begin{align*}
    &\mathbb{E}\left[\sup _{0 \leq s \leq t}\|Z^{(n)}(t, \cdot)\|_{p}^{p} \right] \leq   Ct \mathbb E\left[\sup\limits_{s\in[0,t]}\|Z^{(n-1)}(s)\|_{p}^p \right].
\end{align*}
By taking $t\in(0,T]$ small enough one can follow the proof of \cite{KXZ}~Lemma 5.5 to conclude the existence of unique solution $U \in L^{p}\left(\Omega ; C\left([0, t], L^{p}\right)\right)$ and satisfies \eqref{e.Q02191}.

\end{proof}
\section{Global existence with small initial data and small noise}\label{sec:global}
To prove the global existence part in the main Theorem \ref{t.w10101}, we first recall the truncated stochastic Boussinesq system:
\begin{equation}\label{e.w02231}
\begin{split}
d U-\Delta U d t
&=\varphi\left(\left\|U\right\|_{p}\right)^2 B(U) dt 
+G(U) d t +\left(\mathbf{P}(b\cdot\nabla U) + \varphi\left(\left\|U\right\|_{p}\right)^2\sigma(U)\right) d\mathbb{W}_{t} 
 \\&\nabla \cdot u=0, 
 \\& U(0)=U_{0} ,
\end{split}
\end{equation}
on $[0, \infty) \times \mathbb{T}^{3}$ with $\nabla \cdot u_{0}=0$ and $\int_{\mathbb{T}^{3}} U_{0} d x=0$ a.s., where $\varphi$ is defined as in Section \ref{s.w02271}.
Then it has been shown in Section \ref{s.w02271} that system \eqref{e.w02231} is globally well-posed. Observe that when $\|U\|_{p}\leq \delta_{0} / 2$, the original system \eqref{e.L10261} coincides with this truncated model. Hence, an estimate of the likelihood that $\|U\|_{p}$ exceeds $\delta_{0} / 2$ determines the time of existence for the solution to \eqref{e.L10261}. The global existence part in Theorem \ref{t.w10101} then follows from Markov inequality and the following theorem.
\begin{theorem}\label{t.w02231}
Let $p>5$. Then the global solution $U \in L^{p}\left(\Omega ; C\left([0, \infty), L^{p}\right)\right)$ to \eqref{e.w02231} satisfies
\begin{equation}\label{L.022502}
\mathbb{E}\left[\sup _{s \in[0, \infty)} e^{a s}\|U(s)\|_{p}^{p}+\int_{0}^{\infty} e^{a s} \sum_{i}\left\|\nabla\left(\left|U_{i}(s)\right|^{p / 2}\right)\right\|_{2}^{2} d s\right] \leq C \mathbb{E}\left[\left\|U_{0}\right\|_{p}^{p}\right],
\end{equation}
provided that $a, \delta_{0}, \epsilon_{0}>0$ and $N_{b,2}$ are sufficiently small constants. 
\end{theorem}
We first supply a proof of \eqref{L.022501} in Theorem \ref{t.w10101} by assuming Theorem \ref{t.w02231}.
\begin{proof}[Proof of global existence in Theorem \ref{t.w10101}]
    Let $\mathbb{E}\left\|U_0\right\|_p^p<\delta$. It then follows from Markov inequality and \eqref{L.022502} that 
    \[\mathbb{P}\left(\sup_{s \in[0, \infty)} e^{a s}\|U(s)\|_p^p\geq \frac{\delta_0}{2}\right)\leq \frac{C\delta}{\delta_0}.\]
    Since when $\|U\|_{p}\leq \delta_{0} / 2$ the original system \eqref{e.L10261} coincides with this truncated model, we have 
\begin{align*}
      \mathbb{P}(\tau=\infty) &= \mathbb{P}\left(\sup_{s \in[0, \infty)} \|U(s)\|_p^p\leq \frac{\delta_0}{2}\right)\\
       &\geq \mathbb{P}\left(\sup_{s \in[0, \infty)} e^{a s}\|U(s)\|_p^p\leq \frac{\delta_0}{2}\right)\geq 1- \frac{C\delta}{\delta_0}.
\end{align*}    
    By choosing $\delta$ sufficiently small, we obtain \eqref{L.022501} as desired. 
\end{proof}
\begin{proof}[Proof of Theorem \ref{t.w02231}]
 Let $T>0$. Applying the Itô-Wentzel formula to $F_{i}(t)=e^{a t}\left\|U_{i}(t)\right\|_{p}^{p}$, for a fixed $i \in\{1,2,3,4\}$, we obtain
\begin{equation}
d\left(e^{a t}\left\|U_{i}(t)\right\|_{p}^{p}\right)=a e^{a t}\left\|U_{i}(t)\right\|_{p}^{p} d t+e^{a t} d\left(\left\|U_{i}(t)\right\|_{p}^{p}\right).
\end{equation}
Similar to the Ito expansion as in Section \ref{s.L022301}, we have 
\begin{align}\label{e.w02232}
\begin{split}
    &e^{at}\|U_j(t)\|_p^p + \frac{4(p-1)}{p}\int_0^t e^{as} \left\|\nabla\left(\left|U_j(s)\right|^{p / 2}\right)\right\|_2^2ds\\ 
    &= \|U_{0, j}\|_p^p + p\int_0^t e^{as}\varphi^2 \int_{\mathbb{T}^3}|U_j(s)|^{p-2}U_j(s) B(U)_jdxds+p\int_0^t e^{as}   
    \int_{\mathbb{T}^3} |U_j(s)|^{p-2}U_j(s) G(U)_jdxds\\
    &+ p\int_0^t e^{as}\varphi^2 \int_{\mathbb{T}^3}|U_j(s)|^{p-2}U_j(s)\left((\mathbf{P}(b\cdot\nabla U))_j + 
    \sigma(U)_j \right)dxd\mathbb{W}_s \\
    &+\frac{p(p-1)}{2}\int_0^t e^{as}\varphi^4 \int_{\mathbb{T}^3}|U_j|^{p-2}\|\big(\mathbf{P}\left(b\cdot\nabla U \right)
    \big)_j+ \sigma(U)_j \|_{l^2}^2dxds+a \int_{0}^{t} e^{a s}\left\|U_{i}(s)\right\|_{p}^{p} d s.
\end{split}
\end{align}
Choosing $\bar{r}, r, q, l$ as in \eqref{e.L030403} and \eqref{e.Q02111}, and using integration by parts, one obtains
\begin{align}\label{e.L022301}
\begin{aligned}
&\hspace{0.5cm}p e^{a s} \varphi^{2}\left|\int_{\mathbb{T}^{3}} \left|U_{j}\right|^{p-2} U_{j}B(U)_{j} d x\right| \\
&= p e^{a s} \varphi^{2}\left|\sum_{i=1}^3 \int_{\mathbb{T}^{3}} \partial_{i}\left(\left|U_{j}\right|^{p-2} U_{j}\right)\left(\mathbf{P}\left(U_{i} U\right)\right)_{j} d x\right|\\
&\leq Ce^{a s} \varphi^{2}\left\|\nabla\left(\left|U_j(s)\right|^{p / 2}\right)\right\|_2\left\||U_j|^{p/2-1}\right\|_{\bar{r}}\|U_i\|_r\|U\|_{l}
\\
&\leq C\delta_0 e^{a s} \varphi^{2}\left\|\nabla\left(\left|U_j(s)\right|^{p / 2}\right)\right\|_2\left\||U_j|^{p/2-1}\right\|_{\bar{r}}\|U\|_p,
\end{aligned}
\end{align}
where we used the facts that $r, l\leq p$ and $\varphi\|U\|_p\leq \delta_0$. Now estimates as in \eqref{e.L030402}-\eqref{e.L021404} give
\begin{align}\label{e.L030405}
\begin{aligned}
\left|\int_0^tp e^{a s} \varphi^{2}\int_{\mathbb{T}^{3}} \left|U_{j}\right|^{p-2} U_{j}B(U)_{j} d xds \right|\leq\delta \int_0^t e^{a s} \left\|\nabla\left(\left|U_j(s)\right|^{p / 2}\right)\right\|_2^2 d s+C_\delta \delta_0^\kappa \int_0^t e^{a s}\|U(s)\|_p^p d s,
\end{aligned}
\end{align}
where $\delta>0$ is arbitrary, $\kappa>0$ is a constant depending on $p$ . Assumption \eqref{e.L022205} and estimate as in \eqref{e.L022107} and  yield 
\begin{equation}\label{e.w02272}
\begin{split}
\mathbb{E}\left[\sup _{t \in[0, T]}\left|\int_0^t e^{a s} \varphi^2 \int_{\mathbb{T}^3} |U_j(s)|^{p-2} U_j(s)  \sigma(U)_j(s) d x d \mathbb{W}_s \right|\right] 
&\leq
\frac{1}{4p}\mathbb{E}\sup_{s\in[0, T]}e^{as}\|U_j(s)\|_{p}^p
\\&\quad+
C\epsilon_0 \mathbb{E}\left[\int_0^T e^{a s}\|U(s)\|_p^p d s\right] .
\end{split}
\end{equation}
Assumption \eqref{e.L022205} on the noise also implies
\begin{equation}\label{e.w02271}
\begin{split}
p(p-1)\int_0^t e^{as}\varphi^4 \int_{\mathbb{T}^3}|U_j|^{p-2}\|\sigma(U)_j \|_{l^2}^2dxds&\leq C \epsilon_0^2 \int_0^te^{a s} \|U(s)\|_p^pds.
\end{split}
\end{equation}
 Estimates as in \eqref{e.L022102}, \eqref{e.L121803} and \eqref{e.L022304} yield
\begin{align}\label{e.L022305}
\begin{split}
&p(p-1)\int_0^t e^{as}\varphi^4 \int_{\mathbb{T}^3}|U_j|^{p-2}\|\big(\mathbf{P}\left(b\cdot\nabla U \right)
    \big)_j\|_{l^2}^2dxds\\
    &\leq C\left(N_{b,0} \int_0^t e^{a s} \left\|\nabla\left(\left|U_j(t)\right|^{p / 2}\right)\right\|_2^2 d s+ N_{b,2}\int_0^te^{as}\left\|U(s)\right\|_p^p d s\right).
 \end{split}
\end{align}
Also estimates as in \eqref{e.L022105}, \eqref{e.L022107} and \eqref{e.L0221011} imply 
\begin{align}\label{e.L022307}
\begin{split}
&\mathbb{E}\left[\sup _{t \in[0, T]}\left|\int_0^t e^{a s} \varphi^2 \int_{\mathbb{T}^3} |U_j(s)|^{p-2} U_j(s)  \sigma(U)_j(s) d x d \mathbb{W}_s \right|\right] \\
&\leq\frac{1}{2p}\mathbb{E}\sup_{s\in[0, T]}e^{as}\|U_j(s)\|_{p}^p +  CN_{b,2} \mathbb{E}\left[\int_0^T e^{a s} \|U(s)\|_p^p d s\right] + C N_{b,0}\mathbb{E}\int_0^T e^{a s} \left\|\nabla\left(\left|U_j(t)\right|^{p / 2}\right)\right\|_2^2 d s.
 \end{split}
\end{align}
Combining \eqref{e.L030405}-\eqref{e.L022307}, one has 
\begin{align}\label{e.L022306}
\begin{split}
    &\mathbb{E}\left[\frac{1}{4}\sup_{s\in[0, T]}e^{as}\|U_j(s)\|_{p}^p + L_1\int_0^T e^{as} \left\|\nabla\left(\left|U_j(s)\right|^{p / 2}\right)\right\|_2^2ds\right]\\ 
    &\leq \mathbb{E}\left[\|U_{0, j}\|_p^p +L_2\int_0^T e^{a s} \|U(s)\|_p^pds +  p\left|\int_0^T e^{as}   
    \int_{\mathbb{T}^3} |U_j(s)|^{p-2}U_j(s) G(U)_jdxds\right|\right],
\end{split}
\end{align}
where 
\begin{align}\label{e.L0223010}
\begin{split}
&L_1 = L_1(\delta, b, p) = \frac{4(p-1)}{p} - \delta - C_pN_{b,0},\\
&L_2 = L_2(\delta, b, p, \delta_0, a) = C_\delta \delta_0^\kappa+C\epsilon_0^2 +C_pN_{b,2} +a.
\end{split}
\end{align}
For $j=1,2,3$, the fact $(GU)_j = \mathcal{P}(U_4 e_3))_j$ and the estimate like \eqref{e.L021404} give
\begin{align}\label{e.L022302}
\begin{split}
&p\left|\int_0^Te^{as}\int_{\mathbb{T}^3}|U_j(s)|^{p-2}U_j(s)(GU)_j(s)dxds\right|\\
&\leq \eta \int_0^Te^{as}\left\|\nabla\left(\left|U_j(r)\right|^{p / 2}\right)\right\|_2^2 d r+\eta \int_0^Te^{as}\|U_j(s)\|_p^pds+C_{\eta}\int_0^Te^{as}\left\|U_4(s)\right\|_{p}^p d s,
\end{split}
\end{align}
where $\eta>0$ is arbitrary.  Note that  $(G(U))_4\equiv0$. Hence inequality \eqref{e.L022306} for $j=4$ and the Poincaré   \eqref{e.L031101} imply that 
\begin{align}\label{e.L022308}
\begin{split}
    \mathbb{E}\int_0^T e^{as}\left\|U_4(s)\right\|_{p}^p d s
    &\leq 
    C_{P}\mathbb{E}\int_0^T e^{as} \left\|\nabla\left(\left|U_4(s)\right|^{p / 2}\right)\right\|_2^2ds
    \\&\leq 
    \mathbb{E}\left[\frac{C_{P}}{L_1}\|U_0\|_p^p +\frac{C_{P}L_2}{L_1}\int_0^T e^{a s} \|U(s)\|_p^pds\right].
\end{split}
\end{align}
Combining \eqref{e.L022306}-\eqref{e.L022308}, we arrive at 
\begin{align}\label{e.L022309}
\begin{split}
    \mathbb{E}\left[\frac{1}{4}\sup_{s\in[0, T]}e^{as}\|U(s)\|_{p}^p + L\sum_{j=1}^4\int_0^T e^{as} \left\|\nabla\left(\left|U_j(s)\right|^{p / 2}\right)\right\|_2^2ds\right]\leq C_{\eta, L_1}\mathbb{E}\|U_0\|_p^p ,
\end{split}
\end{align}
where $L = L_1-\eta  - C\left(L_2+ \eta +\frac{C(\eta)L_2}{L_1}\right)$.  In view of \eqref{e.L0223010}, we can first choose $\delta, N_{b,0}, \eta$ small enough, and then choose $\delta_0, \epsilon_0, N_{b, 2}, a$ in $L_2$ sufficiently small to make $L>0$.  We then deduce \eqref{L.022502} by letting $T\to\infty$. 
\end{proof}

\section{Acknowledgements}
WW was partially supported by an AMS-Simons travel grant.

\bibliographystyle{abbrv}
\bibliography{Local_Boussinesq}

\end{document}